\documentclass[a4paper,12pt]{amsart}
\usepackage{graphicx}
\usepackage{amssymb, enumitem, mathtools}
\usepackage{nicematrix}
\usepackage{verbatim} 

\usepackage[alphabetic,initials]{amsrefs}

\usepackage{xcolor}

\usepackage[breaklinks=true,colorlinks=true,linkcolor=blue,citecolor=red,urlcolor=cyan]{hyperref}
\usepackage{orcidlink}
\allowdisplaybreaks
\usepackage[margin=1in]{geometry}

\newcommand \on{\overline{\nabla}}
\newcommand \Ga{\Gamma}
\newcommand \la{\lambda}

\newcommand \br{\mathbb{R}}

\newcommand \Span{\operatorname{Span}}
\newcommand \id{\operatorname{id}}
\newcommand \<{\langle}
\renewcommand \>{\rangle}
\newcommand \ip{\< \cdot, \cdot \>}

\newcommand \cR{\mathcal{R}}

\newcommand \rk{\operatorname{rk}}

\newcommand \tM{\overline{M}}

\newcommand{\grad}{\mbox{grad\,}}
\newcommand{\secf}{\mathrm{II}}
\theoremstyle{plain}
\newtheorem{theorem}{Theorem}
\newtheorem*{theorem*}{Theorem}
\newtheorem*{corollary*}{Corollary}

\newtheorem*{conj*}{Conjecture}
\newtheorem{lemma}{Lemma}
\newtheorem*{lemma*}{Lemma}
\newtheorem{proposition}{Proposition}
\newtheorem*{prop*}{Proposition}

\theoremstyle{definition}

\newtheorem*{definition*}{Definition}
\newtheorem*{construction*}{Construction}
\theoremstyle{remark}
\newtheorem{remark}{Remark}
\newtheorem{example}{Example}

\theoremstyle{plain}
\newtheorem{maintheorem}{Theorem}

\makeatletter
\@namedef{subjclassname@2020}{%
  \textup{2020} Mathematics Subject Classification}
\makeatother

\title[Constant curvature hypersurfaces of products of space forms]{Hypersurfaces with constant sectional curvature of a product of two space forms}

\author[Jihun Kim]{Jihun Kim \orcidlink{0000-0001-5501-1207}}
\address{Department of Mathematics, KU Leuven, Celestijnenlaan 200B -- Box 2400, 3001 Leuven, Belgium}
\email{jihun.kim@kuleuven.be}

\author[Yuri Nikolayevsky]{Yuri Nikolayevsky \orcidlink{0000-0002-9528-1882}}
\address{Department of Mathematical and Physical Sciences, La Trobe University, Melbourne, Victoria, 3086, Australia}
\email{y.nikolayevsky@latrobe.edu.au}

\author[Ruy Tojeiro]{Ruy Tojeiro \orcidlink{0000-0002-6418-8167}}
\address{Instituto de Ciências Matemáticas e de Computação, Universidade de São Paulo, Av. Trabalhador São Carlense 400, São Carlos, 13566-590, Brazil}
\email{tojeiro@icmc.usp.br}

\thanks{The first-named author was partially supported by FWO and FNRS under EOS project G0I2222N. 
\\
\indent The second-named author was partially supported by ARC Discovery grant DP210100951. 
\\
\indent The third-named author was partially supported by Fapesp grant 2022/16097-2. 
}


\begin{document}

\begin{abstract}
We complete the classification of hypersurfaces with constant sectional curvature of dimension greater than or equal to three in the product of two space forms.
\end{abstract}

\maketitle

\section{Introduction}\label{sec:intro}

The investigation of submanifolds with constant sectional curvature of different ambient spaces, from both local and global viewpoints, is a classical and relevant topic in Differential Geometry, which traces back to the nineteenth century with the beautiful theory of surfaces with constant Gaussian curvature in three-dimensional Euclidean space  (see, e.g.,  \cites{Bi,Da}; see also \cite{Gal} and the references therein for more recent results on the subject).

  For hypersurfaces with constant sectional curvature of dimension $n\geq 3$, the classification in space forms is well known  (see, e.g., \cite{DT}). 
  A natural extension of this classification problem is provided by hypersurfaces in a product $\tM^{n+1}=\tM^{k}(\kappa_1)\times \tM^{n-k+1}(\kappa_2)$ of two space forms with curvatures $\kappa_1$ and $\kappa_2$, for which several results are known.
    In case $\kappa_1=0\ne\kappa_2$ and $k=1$, a classification was given in \cite[Theorems~7.1 and 7.2]{MaT}, after the surface case was addressed in \cites{AEG1,AEG2,ALP} from a global point of view. The classification has recently been extended in \cite[Theorems~11 and 12]{CT} to the case $k\ge 2$. In the cases $\kappa_1=\kappa_2=\pm1$, with $k=2$ and $n=3$, a classification has been achieved  in \cite[Theorem~1.1]{LVWY} and \cite[Theorem~1.1]{LVWYhyp}, respectively.
    
This article aims to complete the classification of all 
hypersurfaces with constant sectional curvature of dimension $n\geq 3$ in a product 
$\tM^{n+1}=\tM^{k}(\kappa_1)\times \tM^{n-k+1}(\kappa_2)$  of two space forms, at most one of which is flat. 

The following is a qualitative statement of our classification. A detailed and more precise formulation will be given in Section \ref{ss:classification} as Theorem~\ref{mainthmb}. In part~\ref{it:rough2} below, we view $\tM^{n+1}$ embedded in the pseudo-Euclidean space $\mathbb{R}^{n+3, \sigma}$ with a (flat) metric of index $\sigma =1, 2$, as the product of the standard embeddings of its constant curvature factors as hyperquadrics in the corresponding Euclidean or Lorentzian spaces.

\begin{maintheorem} \label{thm:rough} 
Let $f\colon M^n\to \tM^{n+1}$, $n\geq 3$, be a {\rm(}smooth{\rm)} isometric immersion of a space of constant sectional curvature $M^n = M^n(\kappa)$ in a product $\tM^{n+1}=\tM^{k}(\kappa_1)\times \tM^{n-k+1}(\kappa_2)$ of two space forms.
\begin{enumerate}[label={\rm{(\roman*)}}]
\item\label{it:rough1} If $\kappa_1 \kappa_2 \ne 0$, then $\kappa = \frac{\kappa_1\kappa_2}{\kappa_1+\kappa_2}$ and one of the following holds:
 \begin{enumerate}[label={\rm{(\roman{enumi}.\alph*)}}]
\item\label{it:rough2} Either  $\kappa_1\kappa_2<0$, with $\kappa_1+\kappa_2<0$, or $\kappa_1,\kappa_2<0$, and $f(M^n)$ is an open subset of the intersection of the embedded $\tM^{n+1} \subset \mathbb{R}^{n+3, \sigma}$, $\sigma\in\{1,2\}$, with a degenerate hyperplane;
\item\label{it:rough3} $n=3$, $k = 2$, and $f(M^3)$ is obtained from a hypersurface with constant sectional curvature $\varepsilon/2$ in the product $\tM^2(\varepsilon) \times \tM^2(\varepsilon), \varepsilon = \pm 1$, {\rm(}such hypersurfaces are classified in \cites{LVWY,LVWYhyp}{\rm)} by an affine rescaling procedure.\vspace{.5ex}
\end{enumerate}
\item\label{it:rough4} If one of the factors is flat, then $f$ is as given in \cite[Theorems~7.1 and 7.2]{MaT} for $k=1$ and  {\cite[Theorems~11 and 12]{CT}} for $k\ge2$.
\end{enumerate}
\end{maintheorem}

 Explicit \emph{global} isometric immersions  $f\colon  M^n(\kappa)\to \tM^{k}(\kappa_1)\times \tM^{n-k+1}(\kappa_2)\subset \mathbb{R}^{n+3, \sigma}$, $n\geq~3$, whose images are intersections of  $\tM^{k}(\kappa_1)\times \tM^{n-k+1}(\kappa_2)\subset \mathbb{R}^{n+3, \sigma}$ with a degenerate hyperplane of $\mathbb{R}^{n+3, \sigma}$ as in part~\ref{it:rough2} are provided in Section \ref{examples} (see Examples~\ref{ex:int1} and~\ref{ex:int2}). In Subsection 
\ref{ss:alternative}, it is shown that they have the structure of \emph{doubly rotational} hypersurfaces, and explicit parametrisations of them are given. The rescaling procedure in part~\ref{it:rough3} is described in detail in Section \ref{rescaling}.

As a consequence of Theorem \ref{thm:rough}, there does not exist any hypersurface
with constant sectional curvature and dimension $n\geq 3$ in a product $\tM^{k}(\kappa_1)\times \tM^{n-k+1}(\kappa_2)$ of two spaces of constant sectional curvatures $\kappa_1$  and $\kappa_2$ such that
 $\kappa_1\kappa_2<0$ and  $\kappa_1+\kappa_2\geq 0$. 

 Inspired by the hypersurfaces in part~\ref{it:rough2} of Theorem \ref{thm:rough}, we also obtain the following geometric description of the flat hypersurfaces of $\mathbb{R}^k\times \mathbb{H}^{n-k+1}$, $1\leq k\leq n-3$, which comprise part of those in item~\ref{it:rough4} of Theorem \ref{thm:rough}.

 \begin{maintheorem} \label{thm: flatHrough} 
If $f\colon M^n\to \mathbb{R}^k\times \mathbb{H}^{n-k+1}\subset \mathbb{R}^k\times \mathbb{R}^{n-k+2, 1}=\mathbb{R}^{n+2, 1}$, $1\leq k\leq n-3$,  
is a flat hypersurface, then $f(M^n)$ is an open subset of the intersection of $\mathbb{R}^k\times \mathbb{H}^{n-k+1}$ with an affine hyperplane of $\mathbb{R}^{n+2, 1}$ orthogonal to a light-like vector in $\mathbb{R}^{n+2, 1}$.
\end{maintheorem}

A more detailed version of Theorem \ref{thm: flatHrough}  is given by Theorem \ref{thm: flatH} in Section~\ref{ss:flat}, where the hypersurfaces 
 of  $\mathbb{R}^k\times \mathbb{H}^{n-k+1}$, $1\leq k\leq n-3$, which arise in the way described in the statement, are determined according to the relative position of the light-like vector with respect to the orthogonal decomposition $\mathbb{R}^{n+2, 1}=\mathbb{R}^k\times \mathbb{R}^{n-k+2, 1}$. Theorem \ref{thm: flatH} provides a global version, for $2\leq k\leq n-3$, of the local classification of  flat hypersurfaces of $\mathbb{R}^k\times \mathbb{H}^{n-k+1}$ obtained in \cite[Theorem~12(iii)]{CT}.

\vspace{1ex}
     This article is part of a more general program of classifying all locally symmetric hypersurfaces of a globally symmetric space, which is completed in our paper \cite{KNT}. 

\begin{remark}
    It will follow from \cite[Corollary~1]{KNT} that no hypersurfaces of constant sectional curvature exist in the product of \emph{more than two} space forms (no more than one of which being flat). Thus, the results presented in this paper complete the classification of hypersurfaces of constant sectional curvature in arbitrary products of space forms.
\end{remark}

\section{Preliminaries}\label{sec:hypprod}
In this section, following \cite{LTV}, we summarise the basic equations for hypersurfaces in a product of two space forms. Throughout this paper, we assume that at most one of the factors is flat.

Given an isometric immersion $f\colon M^n \to \tM^{n+1}=\tM^{k}(\kappa_1)\times \tM^{n-k+1}(\kappa_2)$ and a (local) unit normal vector field $\xi$, we define a smooth vector field $\zeta$ on $M^n$, $t\in C^\infty(M^n)$ and $\cR \in \mathrm{End}(TM^n)$ by
\begin{equation}\label{zeta}
\pi_2^* \xi = f_*\zeta + t\xi
\end{equation}
and
\begin{equation}\label{R}
\pi_2^* (f_*X) = f_*(\cR X) + \<X,\zeta\>\xi,
\end{equation}
for all $X\in \Gamma(TM^n)$, which imply that
\begin{equation}\label{zeta2}
\pi_1^* \xi = -f_*\zeta + (1-t)\xi
\end{equation}
and
\begin{equation}\label{R2}
\pi_1^* (f_*X) = f_*((I-\cR) X) - \<X,\zeta\>\xi,
\end{equation}
for all $X\in \Gamma(TM^n)$.
Here $\pi_i$ is the orthogonal projection onto the $i^{th}$--factor of $\tM^{n+1}$, with its derivative $\pi_i^*$, $1\leq i\leq 2$ (apart from orthogonal projections, we use the subscript~$_*$, such as $f_*$, to denote the differential of immersions). It follows from \eqref{R} that $\cR=L^tL$, where $L=\pi_2^*f_*$ and $L^t$ is the transpose of $L$; hence, $\cR$ is a symmetric endomorphism. 

Applying $(\pi_2^*)^2=\pi_2^*$ to tangent and normal vector fields, and using \eqref{zeta} and \eqref{R}, gives
\begin{equation}\label{propR}
    \cR(I-\cR)X = \<X,\zeta\>\zeta,\quad((1-t)I - \cR)\zeta = 0,\quad t(1-t)=\|\zeta\|^2
\end{equation}
for all $X\in\Ga(TM^n)$. By the first equation in \eqref{propR}, for each $x\in M^n$, either $\zeta$ vanishes at $x$ and $T_xM^n$ decomposes orthogonally as $T_xM^n=\ker \cR(x) \oplus \ker (I-\cR(x))$, or  $T_xM^n$ decomposes orthogonally as $T_xM^n=\ker \cR(x) \oplus \ker (I-\cR(x))\oplus \Span\{\zeta(x)\}$. Note, by~\eqref{zeta},~\eqref{zeta2}, and the last equation in \eqref{propR}, that $\zeta(x)=0$ if and only if one of $\pi_i^* \xi(x)$, $i=1,2$, vanishes. 

From $\zeta(x)\in (\ker \cR(x))^\perp$ and \eqref{R}, we have $X\in \ker \cR(x)$ if and only if $(\pi_2^*)f_*X~=~0$.  Similarly, $X\in \ker (I-\cR(x))$ if and only if  $(\pi_1^*)f_*X=0$. It follows easily that
\begin{equation}\label{dimkerR}
    k-1\leq \dim \ker \cR(x)\leq k \,\,\,\,\mbox{and}\,\,\,\,n-k\leq \dim \ker (I-\cR(x))\leq n-k+ 1
    \end{equation}
     at any $x\in M^n$.  Thus, if  $\zeta(x)\neq 0$ and $2\leq k\leq n-1$, in view of the second equation in \eqref{propR}, the eigenvalues of $\cR$ are $0$, $1$ and $r=1-t$, the latter having $\zeta$ as an eigenvector. 
Moreover,  $\cR$ (respectively, $I-\cR$) vanishes identically if and only if $k=n$ and $f(M^n)$ is contained in a slice $\tM^{n}(\kappa_1) \times \{x_2\}$, $x_2\in \tM^{1}(\kappa_2)$
    (respectively, $k=1$ and $f(M^n)$ is contained in a slice  $\{x_1\}\times \tM^{n}(\kappa_2)$, $x_1\in \tM^{1}(\kappa_1)$). 
   
  The \emph{product angle function} of a hypersurface $f\colon M^n\to \tM^k(\kappa_1)\times \tM^{n-k+1}(\kappa_2)$, with a unit normal vector field $\xi$, is defined by $C=\langle P\xi, \xi\rangle$, where  $P$ is the product structure of $\tM^k(\kappa_1)\times \tM^{n-k+1}(\kappa_2)$, given by $P(X_1, X_2)=(X_1, -X_2)$
for tangent vector fields $X_1$ and $X_2$ on $\tM^k(\kappa_1)$ and $\tM^{n-k+1}(\kappa_2)$, respectively.
Thus, from~\eqref{zeta} and~\eqref{zeta2} we obtain
\begin{equation}\label{angle}
C=\langle \pi_1^*\xi-\pi_2^*\xi, \xi\rangle=1-2t=2r-1.
\end{equation}

A hypersurface $f\colon M^n\to \tM^{k}(\kappa_1) \times \tM^{n-k+1}(\kappa_2)$ is said to be \emph{decomposable} if it satisfies one of the following:
\begin{itemize}
    \item $M^n$ is (isometric to) a Riemannian product $N^{k-1}\times \tM^{n-k+1}(\kappa_2)$ and $$f=f_1\times \id\colon N^{k-1}\times \tM^{n-k+1}(\kappa_2)\to \tM^{k}(\kappa_1) \times \tM^{n-k+1}(\kappa_2);$$
    \item $M^n$ is (isometric to) a Riemannian product  $\tM^k(\kappa_1)\times N^{n-k}$ and $$f=\id\times f_2\colon \tM^k(\kappa_1)\times N^{n-k}\to \tM^{k}(\kappa_1) \times \tM^{n-k+1}(\kappa_2),$$
\end{itemize}
where $f_1\colon N^{k-1}\to \tM^{k}(\kappa_1)$ and $f_2\colon N^{n-k}\to \tM^{n-k+1}(\kappa_2)$ are hypersurfaces and $\id$ denotes the respective identity map. We say that the hypersurface $f$ is \emph{locally decomposable} if, for each point $x\in M^n$, there exists an open neighbourhood isometric to a Riemannian product, restricted to which $f$ is decomposable.

The next proposition gives some characterisations of local decomposability of a hypersurface $f\colon M^n\to \tM^{k}(\kappa_1) \times \tM^{n-k+1}(\kappa_2)$.

 \begin{proposition}\label{splits}
     For a hypersurface $f\colon M^{n}\to \tM^{k}(\kappa_1)\times \tM^{n-k+1}(\kappa_2)$, $2\leq k\leq n-1$, the following are equivalent:
     \begin{enumerate}[label={{\rm(\alph*)}}]
         \item $f$ is locally decomposable.
         \item The tangent vector field $\zeta$ vanishes everywhere.
         \item The product angle function $C\equiv \pm 1$.
         \item $\dim\ker \cR = k$ or $\dim\ker (I-\cR)=n-k+1$ everywhere.
     \end{enumerate}
     Moreover, if, in addition, $M^n$ is complete and simply connected, then any of the above conditions implies that $f$ is decomposable.
\end{proposition}
\begin{proof}
The equivalence between (a) and (b) is a particular case of \cite[Proposition~3.3]{MeT1}. By the third equation in \eqref{propR},  $\zeta$ vanishes everywhere if and only if $t=0$ or $t=1$ everywhere. By \eqref{angle}, this holds if and only if $C=1$ or $C=-1$ everywhere, respectively. Thus, (b) and (c) are equivalent. Finally, by the discussion after \eqref{propR}, $\zeta$ vanishes everywhere if and only if $\dim \ker \cR + \dim \ker (I-\cR)=n$ everywhere. The equivalence between (b) and (d) then follows from \eqref{dimkerR}.  The last assertion is also a consequence of \cite[Proposition~3.3]{MeT1}.
\end{proof}

Given a unit normal vector field $\xi$ along a hypersurface $f\colon M^{n}\to \tM^{k}(\kappa_1)\times \tM^{n-k+1}(\kappa_2)$, let $S$ denote the corresponding shape operator. Below we list some differential equations relating $S$, $\cR$, $\zeta$ and $t$ that follow by computing the tangent and normal components in $\on \pi_2^* = \pi_2^* \on$, applied to both tangent and normal vectors, and using Gauss and Weingarten formulae:
\begin{equation}\label{derR}
    (\nabla_X \cR)Y= \<Y, \zeta\>SX+\<SX,Y\>\zeta,
\end{equation}
\begin{equation}\label{derS2}
    \nabla_X\zeta=(tI-\cR)SX
\end{equation}
and 
\begin{equation}\label{dert}
    X(t)=-2\<SX, \zeta\>
\end{equation}
for all $X, Y\in\Ga(TM^n)$. The last equation can be written as 
\begin{equation}\label{dert2}
    \grad t=-2 S \zeta.
\end{equation}

For further use, we record the following facts.
\begin{lemma}\label{lem:grad}
    If a hypersurface $f\colon M^{n}\to \tM^{k}(\kappa_1)\times \tM^{n-k+1}(\kappa_2)$ satisfies $S\cR = \cR S$, then $\zeta$ is {\rm(}locally{\rm)} a gradient vector field.
\begin{proof}
    It suffices to show that the one-form $\zeta^\flat =\<\zeta,\cdot\>$ is closed. Since $S\cR = \cR S$, the endomorphism $A:=(tI-\cR)S$ is symmetric. Then, using~\eqref{derS2}, a direct computation yields
    $d\zeta^\flat (X,Y) = \<AX,Y\> - \<AY,X\> = 0$ for all $X,Y\in\Gamma(TM^n)$. Thus, $d\zeta^\flat = 0$.
\end{proof}
\end{lemma}

The Gauss and Codazzi equations are, respectively,
\begin{gather}
    R(X,Y) = \kappa_1(X\wedge Y - X\wedge\cR Y - \cR X \wedge Y) + (\kappa_1+\kappa_2)(\cR X \wedge \cR Y) + SX \wedge SY,\label{Gauss}\\
    (\nabla_X S)Y  - (\nabla_Y S)X = (\kappa_1I-(\kappa_1+\kappa_2)\cR)(X\wedge Y)\zeta,\label{codazzi}
\end{gather}
where $(X\wedge Y)Z = \<Y,Z\>X - \<X,Z\>Y$ for all tangent vector fields on $M^n$.\vspace{1ex}

In the next statement,  $\rk S$ denotes the rank of the shape operator $S$.

\begin{lemma} \label{lem:rks01}
    Let $f\colon M^{n}\to \tM^{k}(\kappa_1)\times \tM^{n-k+1}(\kappa_2)$, $2\leq k\leq n-1$, $\kappa_1\kappa_2\neq 0$, be a hypersurface. If $M^n$ has constant sectional curvature, then $\rk S\ge2$ everywhere on $M^n$.
    \end{lemma}
\begin{proof}
    Arguing by contradiction, suppose that $\rk S \le 1$ at some $x\in M^n$. Then the Gauss equation \eqref{Gauss} gives
    \begin{equation}\label{Gaussrankone}
    R(X,Y) = \kappa_1(X\wedge Y - X\wedge\cR Y - \cR X \wedge Y) + (\kappa_1+\kappa_2)(\cR X \wedge \cR Y)
    \end{equation}
    for all $X,Y\in T_xM^n$.  Since $2\leq k\leq n-1$, by \eqref{dimkerR} we can choose nonzero
$X_1\in \ker \cR$ and $X_2\in \ker(I-\cR)$, and \eqref{Gaussrankone} gives $R(X_1, X_2)=0$. 
If $\zeta\neq 0$ at $x$, applying \eqref{Gaussrankone}  to $X_i$ and $\zeta$,  $i=1,2$, yields
$0=R(X_1, \zeta)=\kappa_1(1-r)(X_1\wedge \zeta)$ and $0=R(X_2, \zeta)=\kappa_2r (X_2\wedge \zeta)$, a contradiction.
If $\zeta=0$ at $x$, then either $\ker \cR$ or   $\ker (I-\cR)$, say, the former, has dimension at least two. Applying \eqref{Gaussrankone} to linearly independent $X_1, Y_ 1\in \ker \cR$ yields $0=R(X_1, Y_1)=\kappa_1 (X_1\wedge Y_1)$, again a contradiction.  The statement follows.
\end{proof}

In the following, we denote by $\tM^{N,\delta}(\kappa)$ an $N$-dimensional pseudo-Riemannian space form with curvature $\kappa$ and index $\delta$. We write $\tM^{N,0}(\kappa)$ only when it is necessary to emphasise that the metric is Riemannian; otherwise, we use the simpler notation $\tM^{N}(\kappa)$.

We use the fact that $\tM^{N,\delta}(\kappa)$ admits an umbilical isometric immersion into $\tM^{N+1,\delta}(\tilde{\kappa})$ if $\kappa>\tilde{\kappa}$, and into $\tM^{N+1,\delta+1}(\tilde{\kappa})$ if $\kappa<\tilde{\kappa}$. In particular, $\tM^{N,\delta}(\kappa)$ with $\kappa\ne0$ admits an isometric embedding in $\br^{N+1,\delta+\sigma(\kappa)}$, which stands for a (pseudo-)Euclidean space endowed with a flat metric of index $\delta+\sigma(\kappa)$ (when $\delta + \sigma(\kappa) =1$, we take a connected component of the quadric in the following), as
\[
\tM^{N,\delta}(\kappa) = \left\{x\in \br^{N+1,\delta+\sigma(\kappa)}:\<x,x\> = \kappa^{-1}\right\}.
\]
Here, for $\kappa\in\br$, we set $\sigma(\kappa)=1$ if $\kappa<0$ and $\sigma(\kappa)=0$ otherwise.
Therefore, $\tM^{n_1}(\kappa_1)\times \tM^{n_2}(\kappa_2)$ admits a canonical isometric embedding
\[
h\colon \tM^{n_1}(\kappa_1)\times \tM^{n_2}(\kappa_2) \to \br^{N_1,\sigma(\kappa_1)} \times \br^{N_2,\sigma(\kappa_2)} = \br^{N_1+N_2,\eta},
\]
with $\eta= \sigma(\kappa_1)+\sigma(\kappa_2)$, $N_i=n_i+1$ if $\kappa_i\ne0$ and $N_i=n_i$ otherwise, in which case $\tM^{n_i}(\kappa_i)$ stands for $\br^{n_i}$.
Denote by $\tilde{\pi}_i\colon\br^{N_1+N_2,\eta} \to \br^{N_i,\sigma(\kappa_i)}$ the canonical projection for $i=1,2$ (note that $\tilde{\pi}_i^* = \tilde{\pi}_i$). 

From now on, we assume that $\kappa_1\kappa_2\ne0$ and set $n_1=k$ and $n_2=n-k+1$, so that $N_1 = k+1$ and $N_2 = n-k+2$. Then, the normal space of $h$ at each point $z\in \tM^{k}(\kappa_1) \times \tM^{n-k+1}(\kappa_2)$ is spanned by $\tilde{\pi}_1(h(z))$ and $\tilde{\pi}_2(h(z))$, and the second fundamental form of $h$, denoted by $\secf_{h}$, is given as
\[
\secf_h(X,Y)=-\kappa_1\<\tilde\pi_1 X, Y\>\tilde\pi_1\circ h - \kappa_2\<\tilde\pi_2X, Y\>\tilde\pi_2\circ h
\]
for all vector fields $X,Y$ tangent to $\tM^{k}(\kappa_1)\times\tM^{n-k+1}(\kappa_2)$.

Given a hypersurface $f\colon M^n \to \tM^{k}(\kappa_1) \times \tM^{n-k+1}(\kappa_2)$, 
denote $\tilde{F}:=h\circ f$ and write $\kappa_i = \varepsilon_i/R_i^2$, where $\varepsilon_i = 1$ if $\kappa_i>0$ and $\varepsilon_i=-1$ if $\kappa_i<0$. Then, the vector field $\nu_1 = \frac{1}{R_1}\tilde{\pi}_1 \circ \tilde{F}$ is normal to $\tilde{F}$, $\<\nu_1,\nu_1\> = \varepsilon_1$, and 
\(
\tilde{\nabla}_X\nu_1=\frac{1}{R_1}\tilde\pi_1 \tilde{F}_*X
\)
for $X\in\Gamma(TM^n)$,
where $\tilde{\nabla}$ stands for the 
(flat) 
Levi-Civita connection in $\br^{n+3,\eta}$.
Taking the tangent and normal components of $\tilde{\nabla}_X \nu_1$ gives, respectively,
\begin{equation} \label{eq:nu}
S_{\nu_1} = -\frac{1}{R_1}(I-\cR)\,\,\,\,\,\mbox{and}\,\,\,\,\,\,\nabla_X^\perp \nu_1=-\frac{1}{R_1}\<X, \zeta\>h_*\xi.
\end{equation}
Similarly, the vector field $\nu_2 = \frac{1}{R_2}\tilde{\pi}_2 \circ \tilde{F}$ is also normal to $\tilde{F}$, $\<\nu_2,\nu_2\>=\varepsilon_2$, and $\tilde{\nabla}_X \nu_2 = \frac{1}{R_2}\tilde\pi_2 \tilde{F}_*X$. Thus
\begin{equation} \label{eq:nu'}
S_{\nu_2} = -\frac{1}{R_2}\cR\,\,\,\,\,\mbox{and}\,\,\,\,\,\,\nabla_X^\perp \nu_2=\frac{1}{R_2}\<X, \zeta\>h_*\xi.
\end{equation}

\vspace{1ex}
We conclude this section by recording an algebraic result that will be
used later.

Let $V^n$ and $U^{p,q}$ be real vector spaces of dimensions $n$ and  $p+q$, respectively, where
$U^{p,q}$ is endowed with an inner product $\ip$ of signature $(p,q) $. Thus, $\ip$ is positive definite if $q=0$ and Lorentzian if $q=1$. A bilinear form $\omega:V^n \times V^n \to U^{p,q}$  is said to be \emph{flat} with respect to $\ip$ if $\<\omega(X,Y),\omega(Z,W)\> = \<\omega(X,W),\omega(Z,Y)\>$ for every $X,Y,Z,W \in V^n$. 
We denote the nullity subspace of $\omega$ by $\mathcal{N}(\omega)=\{Y\in V^n : \omega(X,Y)=0\;\text{for all}\;X\in V^n\}$, while its image subspace by $\mathcal{I}(\omega)=\Span \{\omega(X,Y) : X,Y\in V^n\}$. 
In particular, when $\omega$ is the second fundamental form of an isometric immersion, the subspace $\mathcal{N}(\omega)$ is called the \emph{relative nullity subspace}, and its dimension $\dim\mathcal{N}(\omega)$ the \emph{index of relative nullity}.
For further details, we refer the reader to \cite[Chapter~4]{DT}.

The lemma below was proved by E.\,Cartan in Chapter~II of \cites{Car19,Car20}, under the assumption in~\eqref{it:pode}, and by J.\,D.\,Moore in \cite[Corollary~$1$]{Moo}, under the assumption in~\eqref{it:symim} (see also Lemmas 4.10 and 4.14  and the more general Lemma~4.20 in \cite{DT}).

\begin{lemma}\label{lem:flatbi}
    Let $\omega: V^n \times V^n \to U^{p,q}$, $\dim U<\dim V$, be a flat symmetric bilinear form. Suppose that one of the following holds{\rm:}
    \begin{enumerate}[label={\rm{(\alph*)}}, ref=\alph*]
        \item\label{it:pode} $q=0${\rm;} 
        \item\label{it:symim} $q=1$ and $\mathcal{I}(\omega) = U^{p,1}$.
    \end{enumerate}
Then $\dim \mathcal{N}(\omega) \ge \dim V - \dim U$.
\end{lemma}

\section{The examples}\label{examples}
In this section, we produce examples of hypersurfaces of dimension $n\ge3$ and constant sectional curvature $\kappa = \frac{\kappa_1\kappa_2}{\kappa_1+\kappa_2}$ of $\tM^k(\kappa_1)\times\tM^{n-k+1}(\kappa_2) \subset \br^{k+1,\sigma(\kappa_1)}\times \br^{n-k+2,\sigma(\kappa_2)}=\br^{n+3,\eta}$, $\kappa_1+\kappa_2<0$, $\eta = \sigma(\kappa_1)+\sigma(\kappa_2)$, that arise as intersections of $\tM^k(\kappa_1)\times\tM^{n-k+1}(\kappa_2)$ with a degenerate hyperplane through the origin in $\br^{n+3,\eta}$.

Throughout, for the image of the standard embedding of the hyperbolic space in the Lorentzian space with codimension $1$, we always take one of the two connected components of the corresponding quadric, without saying this explicitly.

\begin{example}\label{ex:int1}
    Let $2\le k \le n-1$ and let $\tilde{\pi}_1\colon\br^{n+3,1}\to \br^{k+1}$ and $\tilde{\pi}_2\colon \br^{n+3,1}\to\br^{n-k+2,1}$ be orthogonal projections. Define
    \[
    \tM^{k,0}(\kappa_1) = \{x \in \br^{k+1} \, | \, \<x,x\> = \kappa_1^{-1}\},\,\,\, \tM^{n-k+1,0}(\kappa_2) = \{y \in \br^{n-k+2,1} \, | \, \<y,y\> = \kappa_2^{-1}\}
    \]
 and 
 $$\tM^{n+2,1}(\kappa) = \{z \in \br^{n+3,1} \, | \, \<z,z\> = \kappa^{-1}\},$$
  where $\kappa^{-1}=\kappa_1^{-1}+\kappa_2^{-1}$, and where we have $\kappa_2 < 0 < \kappa_1$ and $\kappa_1+\kappa_2 <0$ (and so $\kappa > 0$). 
  Clearly $\tM^{k,0}(\kappa_1) \times \tM^{n-k+1,0}(\kappa_2) \subset \tM^{n+2,1}(\kappa) $.

  Let $w \in \br^{n+3,1}$ be a light-like vector. Denote $w_i = \tilde{\pi}_i(w)$,  for $i=1,2$. We can assume that $\<w_1 ,w_1\> = 1$ and $\<w_2 ,w_2\> = -1$. Let $\br^{n+1} = \Span(w_1 , w_2)^\perp$ and let $\iota\colon \tM^{n,0}(\kappa) \to \br^{n+1}$ be the isometric embedding of the sphere $\tM^{n,0}(\kappa)$ such that $\<\iota(p), \iota(p)\>=\kappa^{-1}$. Define the embedding $\phi\colon\tM^{n,0}(\kappa) \to \br^{n+3,1}$ by $\phi(p) = \iota(p) + \upsilon(p) w$ for $p \in \tM^{n,0}(\kappa)$, where we specify the function $\upsilon\colon \tM^{n,0}(\kappa) \to \br$ a little later. Note that $w\in \Span(w,\iota(p))^\perp$, for all $p \in \tM^{n,0}(\kappa)$, by construction.

  We have $\<\phi(p),\phi(p)\> = \<\iota(p),\iota(p)\> = \kappa^{-1}$, and so $\phi(\tM^{n,0}(\kappa)) \subset \tM^{n+2,1}(\kappa)$. Moreover, since $w$ is light-like and is orthogonal to $\iota(p)$, the embedding $\phi$ is an isometry. To ensure that $\phi(\tM^{n,0}(\kappa)) \subset \tM^{k,0}(\kappa_1)\times\tM^{n-k+1,0}(\kappa_2)$, we choose the function $\upsilon\colon\tM^{n,0}(\kappa) \to \br$ in the definition of $\phi$ in such a way that $\<\tilde{\pi}_1 (\phi(p)),\tilde{\pi}_1 (\phi(p))\>= \kappa_1^{-1}$. 
  This gives $\upsilon(p)= (\kappa_1^{-1}-\<\tilde{\pi}_1 (\iota(p)),\tilde{\pi}_1 (\iota(p))\>)^{1/2}$. Hence
  $$\<\tilde{\pi}_1 (\iota(p)),\tilde{\pi}_1 (\iota(p))\><\<\iota(p),\iota(p)\> = \kappa^{-1}= \kappa_1^{-1}+\kappa_2^{-1} < \kappa_1^{-1},$$ so the right-hand side is well defined. Note that $$\phi(\tM^{n,0}(\kappa)) =(\tM^{k,0}(\kappa_1) \times \tM^{n-k+1,0}(\kappa_2))\cap w^\perp \subset \br^{n+3,1}.$$
\end{example}

\begin{example}\label{ex:int2}
     Let $2\le k \le n-1$ and let $\tilde{\pi}_1\colon\br^{n+3,2}\to \br^{k+1,1}$ and $\tilde{\pi}_2\colon \br^{n+3,2}\to\br^{n-k+2,1}$ be orthogonal projections. For $\kappa_1,\kappa_2<0$, define
    \[
    \tM^{k,0}(\kappa_1) = \{x \in \br^{k+1,1} \, | \, \<x,x\> = \kappa_1^{-1}\},\,\,\, \tM^{n-k+1,0}(\kappa_2) = \{y \in \br^{n-k+2,1} \, | \, \<y,y\> = \kappa_2^{-1}\}
    \]
 and 
 $$\tM^{n+2,1}(\kappa) = \{z \in \br^{n+3,2} \, | \, \<z,z\> = \kappa^{-1}\},$$
  where $\kappa^{-1}=\kappa_1^{-1}+\kappa_2^{-1}$. 
  Clearly $\tM^{k,0}(\kappa_1) \times \tM^{n-k+1,0}(\kappa_2) \subset \tM^{n+2,1}(\kappa) $.

  Let $w\in\br^{n+3,2}$ be such that $\<w,w\>=0$, and denote $w_i = \tilde{\pi}_i(w)$, for $i=1,2$. We can assume that one of the following three possibilities holds:

\begin{enumerate}[label={\rm{(\alph*)}}]
    \item \label{it:ex41} $\<w_1 ,w_1\> = -1$ and $\<w_2 ,w_2\> = 1$;
    \item \label{it:ex42} $\<w_1 ,w_1\> = 1$ and $\<w_2 ,w_2\> = -1$;
    \item \label{it:ex43} $\<w_1 ,w_1\> = 0=\<w_2 ,w_2\>$,
\end{enumerate}
In cases \ref{it:ex41} and \ref{it:ex42}, let $\br^{n+1,1} = \Span(w_1,w_2)^\perp$, and let $\iota\colon \tM^{n,0}(\kappa) \to \br^{n+1,1}$ be an isometric embedding of the hyperbolic space $\tM^{n,0}(\kappa)$ into $\br^{n+1,1}$ such that $ \<\iota(p),\iota(p)\>~=~\kappa^{-1}$. Define the embedding $\phi\colon\tM^{n,0}(\kappa) \to \br^{n+3,2}$ by $\phi(p) = \iota(p) + \upsilon(p) w$ for $p \in \tM^{n,0}(\kappa)$, where the function $\upsilon\colon \tM^{n,0}(\kappa)\to \br$ will be specified later. Note that $w \in \Span(w, \iota(p))^\perp$, for all $p \in \tM^{n,0}(\kappa)$, by construction. 

  We have $\<\phi(p),\phi(p)\> = \<\iota(p),\iota(p)\> = \kappa^{-1}$, and so $\phi(\tM^{n,0}(\kappa)) \subset \tM^{n+2,1}(\kappa)$. Moreover, since $\<w,w\>=0$ and $w$ is orthogonal to $\iota(p)$, the embedding $\phi$ is an isometry. To ensure that $\phi(\tM^{n,0}(\kappa)) \subset \tM^{k,0}(\kappa_1)\times\tM^{n-k+1,0}(\kappa_2)$, we choose the function $\upsilon\colon\tM^{n,0}(\kappa) \to \br$ in the definition of $\phi$ as follows. In case \ref{it:ex41}, define
  \[
  \upsilon(p)= (\kappa_2^{-1}-\<\tilde{\pi}_2 (\iota(p)),\tilde{\pi}_2 (\iota(p))\>)^{1/2},
  \]
  so that  $\<\tilde{\pi}_2 (\phi(p)),\tilde{\pi}_2 (\phi(p))\> = \kappa_2^{-1}$. 
  Note that $\<\tilde{\pi}_1 (\iota(p)),\tilde{\pi}_1 (\iota(p))\>\ge0$, hence
  $$\<\tilde{\pi}_2 (\iota(p)),\tilde{\pi}_2 (\iota(p))\> \le \<\iota(p),\iota(p)\> = \kappa^{-1}= \kappa_1^{-1}+\kappa_2^{-1} < \kappa_2^{-1},$$ so that $\phi$ is well defined. 
  In case \ref{it:ex42}, define
   \[
  \upsilon(p)= (\<\tilde{\pi}_2 (\iota(p)),\tilde{\pi}_2 (\iota(p))\>-\kappa_2^{-1})^{1/2},
  \]
  so that 
  $\<\tilde{\pi}_2 (\phi(p)),\tilde{\pi}_2 (\phi(p))\>= \kappa_2^{-1}$. As
 $\<\tilde{\pi}_2(\iota(p)),\tilde{\pi}_2(\iota(p))\>\geq 0$, it follows that
 $$\<\tilde{\pi}_2 (\iota(p)),\tilde{\pi}_2 (\iota(p))\>-\kappa_2^{-1}>0,$$ which shows that the map $\phi$ is well defined.

Now we consider case \ref{it:ex43}. Since $\<w_1 ,w_1\> = \<w_2 ,w_2\> = 0$,  the subspace $w^\perp \subset \br^{n+3,2}$ has degeneracy $1$ and index $1$.
Take any nondegenerate hyperplane $\br^{n+1,1} \subset w^\perp$, and define an isometric embedding $\iota\colon\tM^{n,0}(\kappa) \to \br^{n+1,1}$ in such a way that $\<\iota(p),\iota(p)\> = \kappa^{-1}$. Now define the embedding $\phi\colon \tM^{n,0}(\kappa) \to \br^{n+3,2}$ by $\phi(p) = \iota(p) + \upsilon(p) w$ for $p \in \tM^{n,0}(\kappa)$, where we specify the function $\upsilon\colon \tM^{n,0}(\kappa) \to \br$ a little later. Again,  $w\in \Span( w, \iota(p))^\perp$, for all $p \in \tM^{n,0}(\kappa)$, by construction, and so $\phi$ is an isometric embedding. Moreover, $\<\phi(p),\phi(p)\>=\<\iota(p),\iota(p)\> =\kappa^{-1}$, hence $\phi(\tM^{n,0}(\kappa)) \subset \tM^{n+2,1}(\kappa)$. Therefore, to have $\phi(\tM^{n,0}(\kappa)) \subset \tM^{k,0}(\kappa_1) \times \tM^{n-k+1,0}(\kappa_2)$, we need to choose the function $\upsilon\colon \tM^{n,0}(\kappa) \to \br$ in the definition of $\phi$ in such a way that $\<\tilde{\pi}_2(\phi(p)), \tilde{\pi}_2(\phi(p))\>= \kappa_2^{-1}$, which gives 
\begin{equation*}
  \upsilon(p) = \tfrac12 \<\tilde{\pi}_2(\iota(p)), w_2\>^{-1} (\kappa_2^{-1} - \<\tilde{\pi}_2(\iota(p)), \tilde{\pi}_2(\iota(p))\>).
\end{equation*}
Note that the right-hand side is well defined. Indeed, suppose  $\<\tilde{\pi}_2(\iota(p)), w_2\>=0$ for some $p \in \tM^{n,0}(\kappa)$. Since $w_2 \in \br^{n-k+2,1}$ is light-like, this implies that $\<\tilde{\pi}_2(\iota(p)), \tilde{\pi}_2(\iota(p))\> \ge 0$. But as $\<\iota(p), w\>=0$, we also have $\<\tilde{\pi}_1(\iota(p)), w_1\>=0$, hence $\<\tilde{\pi}_1(\iota(p)), \tilde{\pi}_1(\iota(p))\> \ge 0$, which contradicts the fact that $\<\iota(p),\iota(p)\> = \kappa^{-1} < 0$. Note that, also in this case, 
 $$\phi(\tM^{n,0}(\kappa)) =(\tM^{k,0}(\kappa_1) \times \tM^{n-k+1,0}(\kappa_2))\cap w^\perp \subset \br^{n+3,2}.$$ 
\end{example}

\vspace{1ex}
\subsection{Examples \ref{ex:int1} and \ref{ex:int2} as doubly rotational hypersurfaces}\label{ss:alternative}

  In this subsection, we show that the hypersurfaces constructed in Examples \ref{ex:int1} and \ref{ex:int2} have the structure of doubly rotational hypersurfaces in the sense of \cite[Example~1]{CT2}. We provide explicit parametrisations of them, based on the local parametrisation obtained in \cite[Theorem~3.1]{CT2}  of all  hypersurfaces $
f\colon M^{n}\rightarrow \tM^{k}(\kappa_1)\times \tM^{n-k+1}(\kappa_2)
$ in \emph{class $\mathcal{A}$}, that is, hypersurfaces satisfying $S\cR =\cR S$.  We briefly recall this parametrisation below when $\kappa_1\kappa_2\neq 0$, which is the case of interest here.

Let $h\colon M_1^{k-1} \to \tM^{k}(\kappa_1)$ and $g \colon  M_2^{n-k} \to  \tM^{n-k+1}(\kappa_2)$, $2 \leq  k \leq n -1$, be hypersurfaces with (local) unit normal vector fields $N^h$ and $N^g$, respectively, and corresponding shape operators $S^h$ and $S^g$.
For $s$ in a (sufficiently small) open interval $\mathcal{O}= (-\delta,\delta)$ with $\delta>0$, consider the families of parallel hypersurfaces 
$
h_s\colon M_1^{k-1}\to \tM^{k}(\kappa_1)
$ 
and 
$
g_s \colon  M_2^{n-k} \to \tM^{n-k+1}(\kappa_2)
$ 
of $h$ and $g$, which are given, as maps into $\mathbb{R}^{k+1, \sigma(\kappa_1)}$ and 
$ \mathbb{R}^{n-k+2,\sigma(\kappa_2)}$, respectively, by
\begin{equation*}\label{h_s}
h_s (x_1) = C_{\varepsilon_1}( s / R_1) h(x_1) + R_1 S_{\varepsilon_1}(s/R_1) N^h(x_1), \,\, \varepsilon_1 = \operatorname{sgn}(\kappa_1), \,\, \kappa_1=\varepsilon_1/R_1^2,
\end{equation*}
and 
\begin{equation*}\label{g_s}
g_s (x_2) = C_{\varepsilon_2}( s / R_2) g(x_2) + R_2 S_{\varepsilon_2}(s/R_2) N^g(x_2), \,\, \varepsilon_2 = \operatorname{sgn}(\kappa_2), \,\, \kappa_2=\varepsilon_2/R_2^2,
\end{equation*}
for all  $x_1 \in M_1^{k-1}$ and $x_2 \in M_2^{n-k}$. Here, for $\varepsilon\in \{-1,1\}$, $C_\varepsilon$ and $S_\varepsilon$ denote the real functions given by 
\begin{equation}\label{cs}
C_\varepsilon(s)=
\begin{cases}
\cos  s, & \text{if } \varepsilon =1,\\
\cosh  s, & \text{if } \varepsilon = -1,
\end{cases}
\qquad\text{and}\qquad
S_\varepsilon(s)=
\begin{cases}
\sin  s, & \text{if } \varepsilon=1,\\
\sinh s, & \text{if } \varepsilon=-1.
\end{cases}
\end{equation}  Now, let $\alpha,\beta \colon \mathcal{O} \to \mathbb{R}$ be smooth functions such that 
\begin{equation}\label{smab}
\alpha(0)=\beta(0)=0,\,\,\,  \alpha'(s), \beta'(s) >0 \,\text{ and } \, (\alpha'(s))^2 + (\beta'(s))^2 =1
\end{equation} 
for all $s \in \mathcal{O}.$
It was shown in  \cite{CT2} that the restriction of the map 
$$
f \colon M^n=\mathcal{O} \times M_1^{k-1} \times M_2^{n-k} \to \tM^{k}(\kappa_1) \times \tM^{n-k+1}(\kappa_2) \subset \mathbb{R}^{k+1, \sigma(\kappa_1)}\times \mathbb{R}^{n-k+2,\sigma(\kappa_2)}=\mathbb{R}^{n+3,\eta},$$
$\eta=\sigma(\kappa_1)+\sigma(\kappa_2)$ (here, $M^n=\mathcal{O} \times M_1^{k-1} \times M_2^{n-k}$ is understood as a differentiable manifold, not as a Riemannian product), given by 
\begin{equation}\label{param2}
f(s,x_1,x_2) = h_{\alpha(s)}(x_1) +  g_{\beta(s)}(x_2),
\end{equation}
to the subset of its regular points, defines a hypersurface in class $\mathcal{A}$, and, conversely, any locally indecomposable hypersurface  $
f\colon M^{n}\rightarrow \tM^{k}(\kappa_1)\times \tM^{n-k+1}(\kappa_2)
$ 
in class $\mathcal{A}$ is given locally in this way. 
It was also shown in  \cite{CT2} that, for $f$ defined by \eqref{param2}, 
$$\ker \cR=TM_1^{k-1}, \quad \ker (I-\cR)=TM_2^{n-k} \quad \mbox{and}\quad \zeta=\alpha'(s)\beta'(s)\frac{\partial}{\partial s},$$
 with $\cR\dfrac{\partial}{\partial s}= (\beta'(s))^2\dfrac{\partial}{\partial s}$,
 so $r=r(s)=(\beta'(s))^2$. Taking into account \eqref{smab}, 
 \begin{equation} \label{abder}
 \beta'(s)=\sqrt{r(s)}\quad \mbox{and}\quad \alpha'(s)=\sqrt{1-r(s)}.
 \end{equation} 
 Moreover, let
 \[N_\alpha^h(s,x_1)=-\frac{\varepsilon_1}{R_1}S_{\varepsilon_1}(\alpha(s)/R_1)h(x_1)+C_{\varepsilon_1}(\alpha(s)/R_1)N^h(x_1)\] and 
 \[ N^g_\beta(s,x_2)=-\frac{\varepsilon_2}{R_2}S_{\varepsilon_2}(\beta(s)/R_2)g(x_2)+C_{\varepsilon_2}(\beta(s)/R_2)N^g(x_2) \]
 be unit normal vectors to $h_{\alpha(s)}$ and $g_{\beta(s)}$ at $x_1\in M_1^{k-1}$ and $x_2\in M_2^{n-k}$, respectively, with corresponding shape operators
 \[ S_\alpha^h(s,x_1)=\frac{1}{R_1}\frac{\varepsilon_1S_{\varepsilon_1}(\alpha(s)/R_1)I+R_1C_{\varepsilon_1}(\alpha(s)/R_1)S^h}{C_{\varepsilon_1}(\alpha(s)/R_1)I-R_1S_{\varepsilon_1}(\alpha(s)/R_1)S^h}\] and 
 \[S_\beta^g(s,x_2)=\frac{1}{R_2}\frac{\varepsilon_2S_{\varepsilon_2}(\beta(s)/R_2)I+R_2C_{\varepsilon_2}(\beta(s)/R_2)S^g}{C_{\varepsilon_2}(\beta(s)/R_2)I-R_2S_{\varepsilon_2}(\beta(s)/R_2)S^g}.\]
 Then 
\begin{equation*}
\xi(x) = -\beta'(s)N_\alpha^h(x_1) + \alpha'(s)N^g_\beta (x_2)
\end{equation*}
is a unit normal vector to $f$ at $x=(s,x_1,x_2)$, and the corresponding shape operator is given by
$$
S|_{T_{x_1}M_1}(x)=-\beta'(s)\, S_\alpha^h(s,x_1),\quad \quad
S|_{T_{x_2}M_2} (x)=  \alpha'(s)  S_\beta^g(s,x_2),
$$
and
$$
 S\dfrac{\partial}{\partial s}(x)=(\beta''(s)\alpha'(s)-\alpha''(s)\beta'(s))\frac{\partial}{\partial s}(x).
 $$
 Therefore, the principal curvatures of $f$ at $x=(s, x_1, x_2)$ are
 \begin{equation}\label{p1}
 \lambda_1(x)=\beta''(s)\alpha'(s)-\alpha''(s)\beta'(s),
 \end{equation}
 \begin{equation}\label{p2}
 \lambda_i(x)=-\frac{\beta'(s)}{R_1}\frac{ \varepsilon_1 S_{\varepsilon_1}(\alpha(s)/R_1) + R_1\lambda^h_i(x_1)C_{\varepsilon_1}(\alpha(s)/R_1)}{C_{\varepsilon_1}(\alpha(s)/R_1)-R_1\lambda^h_i(x_1)S_{\varepsilon_1}(\alpha(s)/R_1)}, \quad 2\leq i\leq k,
  \end{equation}
 and
 \begin{equation}\label{p3}
 \lambda_j(x)=\frac{\alpha'(s)}{R_2}\frac{ \varepsilon_2 S_{\varepsilon_2}(\beta(s)/R_2) + R_2\lambda^g_j(x_2)C_{\varepsilon_2}(\beta(s)/R_2)}{C_{\varepsilon_2}(\beta(s)/R_2)-R_2\lambda^g_j(x_2)S_{\varepsilon_2}(\beta(s)/R_2)},\quad k+ 1\leq j\leq n.
  \end{equation}

 Now let $f\colon M^n(\kappa)\to \tM^{n+1}=\tM^{k}(\kappa_1)\times \tM^{n-k+1}(\kappa_2)$, $n\geq 3$, be an isometric immersion of a space of constant sectional curvature $\kappa$ into a product of two space forms with curvatures $\kappa_1$ and $\kappa_2$, 
 given by either of Examples \ref{ex:int1} or \ref{ex:int2}.  Let $\xi$ be a unit normal vector field along $f$, and let
    $
    \ell\colon \tM^{k}(\kappa_1)\times\tM^{n-k+1}(\kappa_2) \to \br^{k+1,\sigma(\kappa_1)}\times \br^{n-k+2,\sigma(\kappa_2)} = \br^{n+3,\eta}
    $
    be the inclusion, where $\eta = \sigma(\kappa_1)+\sigma(\kappa_2)$. Then
    \[
    \ell(\tM^{k}(\kappa_1)\times\tM^{n-k+1}(\kappa_2)) \subset \tM^{n+2,1}(\tilde{\kappa}),\quad 
    {\tilde\kappa}^{-1} = \kappa_1^{-1}+\kappa_2^{-1}.
    \]
Denote $\bar{f} = \ell\circ f$. Recall that, in both  Examples \ref{ex:int1} and \ref{ex:int2},  $\bar{f}(M^n)$ is contained in the intersection of $\ell(\tM^{k}(\kappa_1)\times\tM^{n-k+1}(\kappa_2))\subset \br^{n+3,\eta}$ with a degenerate hyperplane through the origin orthogonal to a light-like vector $w$. We may assume that $w=\ell_*\xi+c_1 \nu_1+c_2 \nu_2$ for some $c_1, c_2\in \mathbb{R}$. 
Since the unit vector field $U =-|\tilde \kappa|^{1/2}(R_1\nu_1+R_2\nu_2)$  lies in the direction of the position vector field of $\bar{f}$, then $\<w, U\>=0=\<w,w\>$, hence
\[ 1+\varepsilon_1 c_1^2+\varepsilon_2 c_2^2=0\quad \mbox{and}\quad  \varepsilon_1 c_1 R_1+\varepsilon_2c_2 R_2=0.\]
After changing $\xi$ by $-\xi$, if necessary, the preceding equations yield
\[\frac{c_1}{R_1}=\frac{\kappa_1}{|\kappa_1+\kappa_2|^{1/2}}\quad \mbox{and}\quad\frac{c_2}{R_2}=-\frac{\kappa_2}{|\kappa_1+\kappa_2|^{1/2}}.\]
Using \eqref{eq:nu} and \eqref{eq:nu'}, we obtain
\[0=S_w=S+c_1\left(-\frac{1}{R_1}(I-\cR)\right)+c_2\left(-\frac{1}{R_2}\cR\right).\]
It follows that  $S=aI+b\cR$, where
\begin{equation}\label{Sab}
a=\frac{\kappa_1}{|\kappa_1+\kappa_2|^{1/2}}
\,\,\,\,\mbox{and}\,\,\,\,
b=-\frac{\kappa_1+\kappa_2}{|\kappa_1+\kappa_2|^{1/2}}.
\end{equation}
 In particular, $S\cR =\cR S$. 
 Therefore, we may assume that, as a differentiable manifold,  $M^n(\kappa)=\mathcal{O} \times M_1^{k-1} \times M_2^{n-k}$, and that $f$ is parametrised by \eqref{param2}.

 \subsubsection{A parametrisation of Example \ref{ex:int1}}

 Here $\kappa_1\kappa_2<0$ and $\kappa_1+\kappa_2<0$, with, say, $\kappa_2<0<\kappa_1$. Writing 
 $\kappa_1=\frac{1}{R_1^2}$ and  $\kappa_2=-\frac{1}{R_2^2}$ (so that $R_1^2>R_2^2$), then \eqref{Sab} becomes
  $$ a=\frac{R_2}{R_1}\frac{1}{\sqrt{R_1^2-R_2^2}}\,\,\,\,\mbox{and}\,\,\,\,b=
 \frac{\sqrt{R_1^2-R_2^2}}{R_1R_2}.
 $$
 Since  $S=aI+b\cR$, then
$$S|_{\ker(\cR)}=aI=\frac{R_2}{R_1}\frac{1}{\sqrt{R_1^2-R_2^2}}I,\quad
S|_{\ker(I-\cR)}=(a+b)I=\frac{R_1}{R_2}\frac{1}{\sqrt{R_1^2-R_2^2}}I,$$
and
$$S\zeta=(a+br)\zeta=\frac{r(R_1^2-R_2^2)+R_2^2}{R_1R_2\sqrt{R_1^2-R_2^2}}\zeta.$$
Hence, the principal curvatures of $f$ are
\begin{equation}\label{pcs1}
\lambda_1=\frac{r(R_1^2-R_2^2)+R_2^2}{R_1R_2\sqrt{R_1^2-R_2^2}},\quad \lambda_i=\frac{R_2}{R_1}\frac{1}{\sqrt{R_1^2-R_2^2}}, \quad \mbox{and} \quad \lambda_j=\frac{R_1}{R_2}\frac{1}{\sqrt{R_1^2-R_2^2}},
\end{equation}
where $2\leq i\leq k$ and $k+1\leq j\leq n$. 
By \eqref{dert}, together with $t=1-r$, we have
$$r'(s)=2\<S\tfrac{\partial}{\partial s}, \zeta\> =2\sqrt{r(s)(1-r(s))}\frac{R_1^2\,r(s)+R_2^2(1-r(s))}{R_1R_2\sqrt{R_1^2-R_2^2}},$$
whose solutions are 
\[
r(s) = \frac{R_2^2 \tan^2\Big(\frac{s}{\sqrt{R_1^2-R_2^2}}+\tfrac{R_1R_2}{2}C\Big)}{R_1^2 + R_2^2 \tan^2\Big(\frac{s}{\sqrt{R_1^2-R_2^2}}+\tfrac{R_1R_2}{2}C\Big)},
\]
where $C=\tfrac{2}{R_1R_2}\tan^{-1}\Big(\pm\tfrac{R_1}{R_2}\sqrt{\tfrac{r(0)}{1-r(0)}}\Big)$.
After translating the parameter $s \to s+\tfrac{R_1R_2}{2}C\sqrt{R_1^2-R_2^2}$, we may assume that, on a (sufficiently) small interval,
\[
r(s) = \frac{R_2^2 \tan^2\theta(s)}{R_1^2 + R_2^2 \tan^2\theta(s)},
\]
where $\theta(s)=\tfrac{s}{\sqrt{R_1^2-R_2^2}}$. 
For (possibly small) $s>0$, equation \eqref{abder} gives
\begin{equation}\label{abprime}
    \alpha'(s) =\frac{R_1}{\sqrt{R_1^2+R_2^2 \tan^2\theta(s)}}\quad \mbox{and} \quad  \beta'(s)=
     \frac{R_2\tan\theta(s)}{\sqrt{R_1^2+ R_2^2\tan^2\theta(s)}}.
\end{equation}
We first note that $\beta(s) = -R_2\tanh^{-1}(u(s))+c$, where $u(s)=\frac{\sqrt{R_1^2+R_2^2\tan^2\theta(s)}}{\sqrt{R_1^2-R_2^2}}$ and $c\in \mathbb{C}$. But then, since $\sqrt{R_1^2+R_2^2\tan^2\theta(s)}>\sqrt{R_1^2}>\sqrt{R_1^2-R_2^2}$, it follows that $u(s)>1$. Using the identity $\tanh^{-1}(u(s)) = \coth^{-1}(u(s)) + \frac{i\pi}{2}$, without loss of generality, we may write $\beta(s) = -R_2 \coth^{-1}(u(s))+\beta$ for some $\beta\in \br$ (the same argument applies when $s<0$).

We now have
\begin{equation}\label{as3}
      \alpha(s) = R_1 \tan^{-1}\left(\frac{\sqrt{R_1^2-R_2^2}\tan\theta(s)}{\sqrt{R_1^2+R_2^2\tan^2\theta(s)}}\right) + \alpha,\quad  \alpha\in\br,
    \end{equation}
    \begin{equation}\label{bs3}
    \beta(s) = -R_2\coth^{-1}\left(\frac{\sqrt{R_1^2+R_2^2\tan^2\theta(s)}}{\sqrt{R_1^2-R_2^2}}\right)+\beta,\quad \beta\in \br.
\end{equation}
On the other hand, equation~\eqref{as3} implies that $$\cot\left(\frac{\alpha(s)-\alpha}{R_1}\right)=\frac{\sqrt{R_1^2\cot^2\theta(s)+R_2^2}}{\sqrt{R_1^2-R_2^2}},$$ 
and if $R_1\lambda_i^h(x_1)=\cot (\alpha/R_1)$, then
\begin{equation}\label{cotas}
\frac{\sin(\alpha(s)/R_1)+R_1\lambda_i^h(x_1)\cos(\alpha(s)/R_1)}
{\cos(\alpha(s)/R_1)-R_1\lambda_i^h(x_1)\sin(\alpha(s)/R_1)}=-\cot\left(\frac{\alpha(s)-\alpha}{R_1}\right).
\end{equation}
It follows from \eqref{p2}, \eqref{abprime} and \eqref{cotas} that
\begin{eqnarray*}
\lambda_i(x)&=&-\frac{\beta'(s)}{R_1}\frac{\sin(\alpha(s)/R_1)+R_1\lambda_i^h(x_1)\cos(\alpha(s)/R_1)}
{\cos(\alpha(s)/R_1)-R_1\lambda_i^h(x_1)\sin(\alpha(s)/R_1)}\\
&=& \frac{R_2}{R_1}\frac{1}{\sqrt{R_1^2-R_2^2}}.
\end{eqnarray*}
Therefore, the principal curvatures $\lambda_i$ of
$f$ with respect to the unit normal vector field $\xi$ are as in \eqref{pcs1} if the principal curvatures $\lambda_i^h$ of $h$ are all equal to $\frac{1}{R_1}\cot (\alpha/R_1)$ for $2\leq i\leq k$.
The latter means that  $h$ is a geodesic hypersphere of radius $\alpha/R_1$
in $\tM^k(\kappa_1)$.

Similarly, from \eqref{bs3} we obtain
\[
\coth\left(\frac{\beta(s)-\beta}{R_2}\right)=-\frac{\sqrt{R_1^2+ R_2^2\tan^2\theta(s)}}{\sqrt{R_1^2-R_2^2}},
\]
and if $R_2\lambda_j^g(x_2)=\coth(\beta/R_2)$, then
\begin{equation}\label{cothbs1}
\frac{- \sinh(\beta(s)/R_2) + R_2\lambda^g_j(x_2)\cosh(\beta(s)/R_2)}{\cosh(\beta(s)/R_2)-R_2\lambda^g_j(x_2)\sinh(\beta(s)/R_2)}
=-\coth\left(\frac{\beta(s)-\beta}{R_2}\right)
\end{equation}
It follows from \eqref{p3},  \eqref{abprime} and \eqref{cothbs1} that
\begin{eqnarray*}
\lambda_j(x)&=&\frac{\alpha'(s)}{R_2}\frac{- \sinh(\beta(s)/R_2) +R_2\lambda^g_j(x_2)\cosh(\beta(s)/R_2)}{\cosh(\beta(s)/R_2)-R_2\lambda^g_j(x_2)\sinh(\beta(s)/R_2)}
\\
&=& \frac{R_1}{R_2}\frac{1}{\sqrt{R_1^2-R_2^2}}, \quad k+1\leq j\leq n.
\end{eqnarray*}
Hence, the principal curvatures $\lambda_j$ of
$f$ with respect to the unit normal vector field $\xi$ are as in \eqref{pcs1} if the principal curvatures $\lambda_j^g$ of $g$ are all equal to $\frac{1}{R_2}\coth (\beta/R_2)$ for $k+1\leq j\leq n$.
The latter means that  $g$ is a geodesic hypersphere of radius $\beta/R_2$ in $\tM^{n-k+1}(\kappa_2)$.

Therefore, the hypersurface in Example \ref{ex:int1} can be parametrised as in \eqref{param2}, with $h$ and $g$ being geodesic hyperspheres of  $\tM^k(\kappa_1)$ and $\tM^{n-k+1}(\kappa_2)$, respectively, and the functions $\alpha(s)$ and $\beta(s)$ being given by 
\eqref{as3} and \eqref{bs3}, respectively. Thus, it is a doubly rotational hypersurface of $\tM^k(\kappa_1)\times\tM^{n-k+1}(\kappa_2)$
in the sense of \cite[Example~1]{CT2}.

To recover the description of $f$ as in Example \ref{ex:int1}, we first apply~\eqref{as3} and~\eqref{bs3} to obtain the following:
\[\begin{array}{l}
h_{\alpha(s)}(x_1)=\cos(\alpha(s)/R_1)h(x_1)+R_1\sin(\alpha(s)/R_1)N^h(x_1)\\
\hspace*{9ex}=\cos\left(\frac{\alpha(s)-\alpha}{R_1}\right)(\cos(\alpha/R_1)h(x_1)+R_1\sin(\alpha/R_1)N^h(x_1))\\
\hspace*{12ex}+R_1\sin\left(\frac{\alpha(s)-\alpha}{R_1}\right)(-\frac{1}{R_1}\sin(\alpha/R_1)h(x_1)+\cos(\alpha/R_1)N^h(x_1))\vspace{1ex}\\
\hspace*{9ex}= \frac{\sqrt{R_1^2\cos^2\theta(s)+R_2^2\sin^2\theta(s)}}{R_1}(\cos(\alpha/R_1)h(x_1)+R_1\sin(\alpha/R_1)N^h(x_1))
\vspace{1ex}\\\hspace*{11ex}
+ 
\sqrt{R_1^2-R_2^2}\,\sin\theta(s)(-\frac{1}{R_1}\sin(\alpha/R_1)h(x_1)+\cos(\alpha/R_1)N^h(x_1)),
\end{array}
\]
\[\begin{array}{l}
g_{\beta(s)}(x_2)=\cosh(\beta(s)/R_2)g(x_2)+R_2\sinh(\beta(s)/R_2)N^g(x_2)\\
\\
\hspace*{9ex}=\cosh\left(\frac{\beta(s)-\beta}{R_2}\right)(\cosh(\beta/R_2)g(x_2)+R_2\sinh(\beta/R_2)N^g(x_2))\\
\hspace*{12ex}+R_2\sinh\left(\frac{\beta(s)-\beta}{R_2}\right)(\frac{1}{R_2}\sinh(\beta/R_2)g(x_2)+\cosh(\beta/R_2)N^g(x_2))\vspace{1ex}\\
\hspace*{9ex}= \frac{\sqrt{R_1^2\cos^2\theta(s)+R_2^2\sin^2\theta(s)}}{R_2}(\cosh(\beta/R_2)g(x_2)+R_2\sinh(\beta/R_2)N^g(x_2))
\vspace{1ex}\\\hspace*{11ex}
-
\sqrt{R_1^2-R_2^2}\,\cos\theta(s)( \frac{1}{R_2}\sinh(\beta/R_2)g(x_2)+\cosh(\beta/R_2)N^g(x_2)).
\end{array}
\]
Note that the image of the map $x_1\mapsto (\cos(\alpha/R_1)h(x_1)+R_1\sin(\alpha/R_1)N^h(x_1))$  is a point $P_0\in \tM^k(\kappa_1)$. Also, the image of the map $x_2\mapsto (\cosh(\beta/R_2)g(x_2)+R_2\sinh(\beta/R_2)N^g(x_2))$ is a point  $Q_0\in \tM^{n-k+1}(\kappa_2)$. 
Let $\{e_1, \ldots, e_{k+1}\}$ and $\{f_1, \ldots, f_{n-k+2}\}$ be orthonormal and pseudo-orthonormal bases of $\mathbb{R}^{k+1}\supset \tM^k(\kappa_1)$ and $\mathbb{R}^{n-k+2,1}\supset \tM^{n-k+1}(\kappa_2)$, respectively, such that $P_0=R_1e_1$ and $Q_0=R_2f_1$. 
Then, the hypersurface can be parametrised by the map $f\colon M^n=\mathcal{O}\times \mathbb{S}^{k-1}\times \mathbb{S}^{n-k}\to\tM^k(\kappa_1)\times  \tM^{n-k+1}(\kappa_2) $, given by
\[
\begin{array}{l}
f(s,x_1, x_2)=\sqrt{R_1^2\cos^2\theta(s)+R_2^2\sin^2\theta(s)}e_1+\sqrt{R_1^2-R_2^2}\,\sin\theta(s)\, i(x_1)\vspace{1ex}\\
\hspace*{13ex}+\sqrt{R_1^2\cos^2\theta(s)+R_2^2\sin^2\theta(s)}f_1-
\sqrt{R_1^2-R_2^2} \cos \theta(s)\, j(x_2),
\end{array}
\]
where $i\colon \mathbb{S}^{k-1}\to \mathbb{R}^k=\{e_1\}^\perp$ and $j\colon \mathbb{S}^{n-k}\to \mathbb{R}^{n-k+1}=\{f_1\}^\perp$ are inclusions.

Finally, we observe that the map 
$\iota\colon  M^n\to \mathbb{R}^k\times  \mathbb{R}^{n-k+1}= \mathbb{R}^{n+1}$,
given by
\[\iota(s, x_1, x_2)=\sqrt{R_1^2-R_2^2}\,(\sin\theta(s)\, i(x_1)- \cos \theta(s)\, j(x_2)),\]
parametrises a sphere of curvature $\kappa$ given by $\kappa^{-1}=\kappa_1^{-1}+\kappa_2^{-1}$ in $\mathbb{R}^{n+1}$, and that 
$f(x)=\iota(x)+v(x)w$
for every $x=(s,x_1, x_2)\in M^n$, where $w$ is the light-like vector $w=e_1+f_1\in \mathbb{R}^{2,1}=(\mathbb{R}^{n+1})^\perp$ and $v(x)=\sqrt{R_1^2\cos^2\theta(s)+R_2^2\sin^2\theta(s)}$. In particular, $f(M^n)$ is contained in the intersection of $\tM^k(\kappa_1)\times \tM^{n-k+1}(\kappa_2)\subset \mathbb{R}^{n+3, 1}$ with the degenerate hyperplane of $\mathbb{R}^{n+3, 1}$ through the origin orthogonal to $w$. 

 \subsubsection{Parametrisations for the Examples \ref{ex:int2}} \label{alternative}
 
  In these examples, both $\kappa_1$ and $\kappa_2$ are negative; writing $\kappa_1=-\frac{1}{R_1^2}$ and  $\kappa_2=-\frac{1}{R_2^2}$, then \eqref{Sab} becomes
  $$ a=\frac{R_2}{R_1}\frac{1}{\sqrt{R_1^2+R_2^2}}\,\,\,\,\mbox{and}\,\,\,\,b=
 -\frac{\sqrt{R_1^2+R_2^2}}{R_1R_2}.
 $$
 Then
 $$S|_{\ker(\cR)}=aI=\frac{R_2}{R_1}\frac{1}{\sqrt{R_1^2+R_2^2}}I, \quad
S|_{\ker(I-\cR)}=(a+b)I=-\frac{R_1}{R_2}\frac{1}{\sqrt{R_1^2+R_2^2}}I,$$
and
$$S\zeta=(a+br)\zeta=\frac{r(R_1^2+R_2^2)-R_2^2}{R_1R_2\sqrt{R_1^2+R_2^2}}\zeta,$$
hence the principal curvatures of $f$ are
\begin{equation}\label{pcs}
\lambda_1=\frac{r(R_1^2+R_2^2)-R_2^2}{R_1R_2\sqrt{R_1^2+R_2^2}},\quad \lambda_i=\frac{R_2}{R_1}\frac{1}{\sqrt{R_1^2+R_2^2}}, \quad \mbox{and}\quad \lambda_j=-\frac{R_1}{R_2}\frac{1}{\sqrt{R_1^2+R_2^2}},
\end{equation}
where $2\leq i\leq k$ and $k+1\leq j\leq n$.
Now, from \eqref{dert} we have
\begin{equation} \label{ode}
r'(s)=2\< S\tfrac{\partial}{\partial s}, \zeta\> =2\sqrt{r(s)(1-r(s))}\,\frac{R_1^2r(s)-R_2^2(1-r(s))}{R_1R_2\sqrt{R_1^2+R_2^2}}.
\end{equation}
The solutions of \eqref{ode} are the constant function 
\begin{equation}\label{constsol}r(s)=\frac{R_2^2}{R_1^2+R_2^2},
\end{equation}
as well as
\[
r(s) = \frac{R_2^2\tanh^2\Big(\tfrac{s}{\sqrt{R_1^2+R_2^2}}+\tfrac{R_1R_2}{2}C\Big)}{R_1^2+R_2^2\tanh^2\Big(\tfrac{s}{\sqrt{R_1^2+R_2^2}}+\tfrac{R_1R_2}{2}C\Big)},
\]
where $C = \tfrac{2}{R_1R_2}\tanh^{-1}\Big(\pm\tfrac{R_1}{R_2}\sqrt{\tfrac{r(0)}{1-r(0)}}\Big)$. After a translation of the parameter $s$, we may assume that
\[
r(s) = \frac{R_2^2\tanh^2 \theta(s)}{R_1^2+R_2^2\tanh^2 \theta(s)},
\]
where $\theta(s)=\tfrac{s}{\sqrt{R_1^2+R_2^2}}$.
From \eqref{abder}, assuming (possibly small) $s> 0$, we obtain 
\begin{equation}\label{derab2}
\alpha'(s)=\frac{R_1}{\sqrt{R_1^2+R_2^2\tanh^2\theta(s)}}\quad \mbox{and}\quad
\beta'(s)= \frac{R_2\tanh\theta(s)}{\sqrt{R_1^2+R_2^2\tanh^2\theta(s)}}.
\end{equation}
 Therefore,
\begin{equation}\label{as}
\alpha(s)=R_1\tanh^{-1}\left(\frac{\sqrt{R_1^2+R_2^2}\tanh\theta(s)}{\sqrt{R_1^2+R_2^2\tanh^2\theta(s)}}\right)+\alpha, \quad \alpha\in \mathbb{R},
\end{equation}
and 
\begin{equation}\label{bs}
\beta(s)=R_2\tanh^{-1}\left(\frac{\sqrt{R_1^2 + R_2^2\tanh^2
\theta(s)}}{\sqrt{R_1^2+R_2^2}}\right)+\beta, \quad \beta\in \mathbb{R}.
\end{equation}
From \eqref{as} we obtain
\begin{equation*}\label{cothas}
\coth\left(\frac{\alpha(s)-\alpha}{R_1}\right)=\frac{\sqrt{R_1^2+R_2^2\tanh^2\theta(s)}}{\sqrt{R_1^2+R_2^2}\tanh\theta(s)
},
\end{equation*}
and if $R_1\lambda_i^h(x_1)=\coth (\alpha/R_1)$, then
\begin{equation}\label{prev}
\frac{-\sinh(\alpha(s)/R_1)+R_1\lambda_i^h(x_1)\cosh(\alpha(s)/R_1)}
{\cosh(\alpha(s)/R_1)-R_1\lambda_i^h(x_1)\sinh(\alpha(s)/R_1)}=-\coth\left(\frac{\alpha(s)-\alpha}{R_1}\right).
\end{equation}
It follows from \eqref{p2}, \eqref{derab2} and \eqref{prev} that
\begin{eqnarray*}
\lambda_i(x)&=&-\frac{\beta'(s)}{R_1}\frac{-\sinh(\alpha(s)/R_1)+R_1\lambda_i^h(x_1)\cosh(\alpha(s)/R_1)}
{\cosh(\alpha(s)/R_1)-R_1\lambda_i^h(x_1)\sinh(\alpha(s)/R_1)}\\
&=& \frac{R_2}{R_1}\frac{1}{\sqrt{R_1^2+R_2^2}}.
\end{eqnarray*}
Therefore, the principal curvatures $\lambda_i$ of
$f$ given by \eqref{param2} with respect to the unit normal vector field $\xi$ are as in \eqref{pcs} if the principal curvatures $\lambda_i^h$ of $h$ are all equal to $\frac{1}{R_1}\coth (\alpha/R_1)$ for $2\leq i\leq k$.
The latter means that  $h$ is a geodesic hypersphere of radius $\alpha/R_1$
in $\tM^k(\kappa_1)$.

Similarly, from \eqref{bs} we find
\begin{equation}\label{cothbs}
\tanh\left(\frac{\beta(s)-\beta}{R_2}\right)=\frac{\sqrt{R_2^2\tanh^2\theta(s)+R_1^2}}{\sqrt{R_1^2+R_2^2}},
\end{equation}
and if $R_2\lambda_j^g(x_2)=\tanh (\beta/R_2)$, then
\begin{equation}\label{tanh}\frac{- \sinh(\beta(s)/R_2) +R_2\lambda^g_j(x_2)\cosh(\beta(s)/R_2)}{\cosh(\beta(s)/R_2)-R_2\lambda^g_j(x_2)\sinh(\beta(s)/R_2)}
=-\tanh\left(\frac{\beta(s)-\beta}{R_2}\right)
\end{equation}
It follows from \eqref{p3},  \eqref{derab2} and \eqref{tanh} that
\begin{eqnarray*}
\lambda_j(x)&=&\frac{\alpha'(s)}{R_2}\frac{- \sinh(\beta(s)/R_2) +R_2\lambda^g_j(x_2)\cosh(\beta(s)/R_2)}{\cosh(\beta(s)/R_2)-R_2\lambda^g_j(x_2)\sinh(\beta(s)/R_2)}
\\
&=& -\frac{R_1}{R_2}\frac{1}{\sqrt{R_1^2+R_2^2}}, \quad k+1\leq j\leq n.
\end{eqnarray*}
Hence, the principal curvatures $\lambda_j$ of
$f$ given by \eqref{param2} with respect to the unit normal vector field $\xi$ are as in \eqref{pcs} if the principal curvatures $\lambda_j^g$ of $g$ are all equal to $\frac{1}{R_2}\tanh (\beta/R_2)$ for $k+1\leq j\leq n$.
The latter means that  $g$ is an equidistant hypersurface in $\tM^{n-k+1}(\kappa_2)$ at a distance $\beta/R_2$ of a totally geodesic hypersurface.

Therefore, the hypersurface  can be parametrised as in \eqref{param2} with $h$ being a geodesic hypersphere of  $\tM^k(\kappa_1)$ and $g$ being an equidistant hypersurface of $\tM^{n-k+1}(\kappa_2)$, and the functions $\alpha(s)$ and $\beta(s)$  being given by 
\eqref{as} and \eqref{bs}, respectively. Thus, it is a doubly rotational hypersurface of $\tM^k(\kappa_1)\times\tM^{n-k+1}(\kappa_2)$
in the sense of \cite[Example~1]{CT2}.

Now note that
\[\begin{array}{l}
h_{\alpha(s)}(x_1)= \cosh(\alpha(s)/R_1)h(x_1) + R_1\sinh(\alpha(s)/R_1)N^h(x_1)\\
\hspace*{9ex}= \frac{\sqrt{R_1^2\cosh^2\theta(s)+R_2^2\sinh^2\theta(s)}}{R_1}(\cosh(\alpha/R_1)h(x_1)+R_1\sinh(\alpha/R_1)N^h(x_1))
\vspace{1ex}\\\hspace*{12ex}
+ 
\sqrt{R_1^2+R_2^2}\,\sinh\theta(s)(\frac{1}{R_1}\sinh(\alpha/R_1)h(x_1)+\cosh(\alpha/R_1)N^h(x_1)),
\end{array}
\]
and, similarly,
\[\begin{array}{l}
g_{\beta(s)}(x_2)=\cosh(\beta(s)/R_2)g(x_2)+R_2\sinh(\beta(s)/R_2)N^g(x_2)\\
\\
\hspace*{9ex}= \frac{\sqrt{R_1^2+R_2^2}}{R_2}\cosh\theta(s)(\cosh(\beta/R_2)g(x_2)+R_2\sinh(\beta/R_2)N^g(x_2))
\vspace{1ex}\\\hspace*{11ex}
+
\sqrt{R_1^2\cosh^2\theta(s)+R_2^2\sinh^2\theta(s)}( \frac{1}{R_2}\sinh(\beta/R_2)g(x_2)+\cosh(\beta/R_2)N^g(x_2)).
\end{array}
\]
Note also that the image of the map $x_1\mapsto (\cosh(\alpha/R_1)h(x_1)+R_1\sinh(\alpha/R_1)N^h(x_1))$  is a point $P_0={R_1}e_1\in \tM^k(\kappa_1)$. Therefore, the map
$ x_1\mapsto \frac{1}{R_1}(\sinh(\alpha/R_1)h(x_1)+R_1\cosh(\alpha/R_1)N^h(x_1))$ 
parametrises an inclusion $i\colon\mathbb{S}^{k-1}\to \mathbb{R}^{k}=\Span(e_1)^\perp$.
On the other hand,  the image of the map $x_2\mapsto \rho(x_2)=(\cosh(\beta/R_2)g(x_2)+R_2\sinh(\beta/R_2)N^g(x_2))$ is a totally geodesic hypersurface of $\tM^{n-k+1}(\kappa_2)$, hence $j=\frac{1}{R_2}\rho$ gives an inclusion
$j\colon \mathbb{H}^{n-k}\to \mathbb{R}^{n-k+1,1}$, where $\mathbb{R}^{n-k+1,1}$ is the linear hyperplane containing $\rho(\tM^{n-k})$. Let  $f_1$ span the orthogonal complement of $\mathbb{R}^{n-k+1,1}$ in $\mathbb{R}^{n-k+2,1}$. 
Then, the hypersurface can be parametrised by the map $f\colon M^n=\mathcal{O}\times \mathbb{S}^{k-1}\times \mathbb{H}^{n-k}\to\tM^k(\kappa_1)\times  \tM^{n-k+1}(\kappa_2) $, given by
\begin{equation}\label{alt}
\begin{array}{l}
f(s,x_1, x_2)=\sqrt{R_1^2\cosh^2\theta(s)+R_2^2\sinh^2\theta(s)}e_1+\sqrt{R_1^2+R_2^2}\,\sinh\theta(s)\, i(x_1) \vspace{1ex}\\
\hspace*{13ex}+\sqrt{R_1^2\cosh^2\theta(s)+R_2^2\sinh^2\theta(s)}f_1+
\sqrt{R_1^2+R_2^2} \cosh \theta(s)\, j(x_2).
\end{array}
\end{equation}
The map 
$\iota\colon  M^n\to \mathbb{R}^k\times  \mathbb{R}^{n-k+1,1}= \mathbb{R}^{n+1,1}$,
given by
\[\iota(s, x_1, x_2)=\sqrt{R_1^2+R_2^2}\,(\sinh\theta(s)\, i(x_1)+ \cosh \theta(s)\, j(x_2)),\]
parametrises a hyperbolic space of curvature $\kappa$ given by $\kappa^{-1}=\kappa_1^{-1}+\kappa_2^{-1}$ in $\mathbb{R}^{n+1,1}$, and
\(f(x)=\iota(x)+v(x)w\)
for every $x=(s,x_1, x_2)\in M^n$, where $w$ is the light-like vector $w=e_1+f_1\in \mathbb{R}^{2,1}=(\mathbb{R}^{n+1,1})^\perp$ and $v(x)=\sqrt{R_1^2\cosh^2\theta(s)+R_2^2\sinh^2\theta(s)}$. In particular, $\langle f(x), w\rangle =0$ for every $x=(s,x_1, x_2)\in M^n$, so $f(M^n)$ is an open subset of the intersection of $\tM^k(\kappa_1)\times \tM^{n-k+1}(\kappa_2)\subset \mathbb{R}^{n+3,2}$ with the hyperplane of $\mathbb{R}^{n+3,2}$ through the origin orthogonal to $w$. Since $\langle e_1, e_1\rangle=-1$ and $\langle f_1, f_1\rangle=1$, \eqref{alt} parametrises a hypersurface $f\colon M^n(\kappa)\to \tM^k(\kappa_1)\times \tM^{n-k+1}(\kappa_2)$ as in the case~\ref{it:ex41} of Example~\ref{ex:int2}.

We observe that if $r(s)$ is a solution of \eqref{ode}, then $t(s)=1-r(s)$ is a solution of the equation obtained from  \eqref{ode} by switching the roles of $R_1$ and $R_2$. The hypersurface obtained as before from such a solution corresponds to the one in the case~\ref{it:ex42} of Example~\ref{ex:int2}.

For the constant solution \eqref{constsol}, the principal curvatures of $f$  are 
\begin{equation}\label{pcsconst}
\lambda_1=0, \quad \lambda_i=\frac{R_2}{R_1}\frac{1}{\sqrt{R_1^2+R_2^2}},  \quad \mbox{and}\quad  \lambda_j=-\frac{R_1}{R_2}\frac{1}{\sqrt{R_1^2+R_2^2}}, 
\end{equation}
where $ 2\leq i\leq k$ and $k+1\leq j\leq n$. From \eqref{abder} we obtain
\begin{equation}\label{derab}
\alpha'(s)=\frac{R_1}{\sqrt{R_1^2+R_2^2}} \quad \mbox{and}\quad \beta'(s)=\frac{R_2}{\sqrt{R_1^2+R_2^2}},
\end{equation}
hence
$$
\alpha(s)=\frac{R_1s}{\sqrt{R_1^2+R_2^2}} \quad \mbox{and}\quad \beta(s)=\frac{R_2s}{\sqrt{R_1^2+R_2^2}},
$$
after a translation of the parameter $s$.
Comparing \eqref{pcsconst} with \eqref{p1}--\eqref{p3} and using \eqref{derab}, it follows that
$h$ and $g$ are horospheres of $\tM^{k}(\kappa_1)$
and $\tM^{n-k+ 1}(\kappa_2)$ with principal curvatures $R_1^{-1}$ and $R_2^{-1}$, respectively.

We can assume that $h$ and $g$ are maps 
$ h\colon  \mathbb{R}^{k-1}\to \tM^{k}(\kappa_1)\subset \mathbb{R}^{2,1}\times \mathbb{R}^{k-1}=\mathbb{R}^{k+1,1}$
and 
$g\colon  \mathbb{R}^{n-k}\to \tM^{n-k+1}(\kappa_2)\subset \mathbb{R}^{2,1}\times \mathbb{R}^{n-k}=\mathbb{R}^{n-k+2,1}$
given by
\begin{equation}\label{hg}
h(x_1)=R_1(1+\|x_1\|^2, 1 , x_1)\quad \mbox{and}\quad g(x_2)=R_2(1+\|x_2\|^2, 1 , x_2),
\end{equation}
 respectively, where the first two coordinates in the expressions of $h$ and $g$ in \eqref{hg} are with respect to pseudo-orthonormal bases $\{e_1, e_2\}$ and $\{f_1, f_2\}$ of the corresponding Lorentzian planes such that $\langle e_1, e_1\rangle=\langle f_1, f_1\rangle=\langle e_2, e_2\rangle=\langle f_2, f_2\rangle=0$ and
 $\langle e_1, e_2\rangle=-1/2=\langle f_1, f_2\rangle$. The vector fields
 $$N^h(x_1)=(-1+\|x_1\|^2, 1, \, x_1)\quad \mbox{and}\quad  N^g(x_2)= (-1+\|x_2\|^2, 1, \, x_2)
$$
are unit normal vector fields along $h$ and $g$, respectively. The hypersurface $f$ can be parametrised by 
$$f\colon M^n=\mathcal{O}\times \mathbb{R}^{k-1}\times \mathbb{R}^{n-k}\to \tM^{k}(\kappa_1)\times \tM^{n-k+1}(\kappa_2)\subset \mathbb{R}^{k+1,1}\times \mathbb{R}^{n-k+2,1}=\mathbb{R}^{n-k+3,2},$$
\begin{equation}\label{casec}
f(x)=\cosh\theta(s)h(x_1)+R_1\sinh\theta(s)
N^h(x_1)+\cosh\theta(s)g(x_2)+R_2\sinh\theta(s)
N^g(x_2),
\end{equation}
for every $x=(s, x_1, x_2)\in M^n$, where $\theta(s)=\frac{s}{\sqrt{R_1^2+R_2^2}}$.
This corresponds to the isometric immersion $f\colon M^n(\kappa)\to \tM^k(\kappa_1)\times \tM^{n-k+1}(\kappa_2)$ given by Example \ref{ex:int2}\ref{it:ex43}, and describes it as a doubly rotational hypersurface. Notice that $w=R_2e_1-R_1f_1$ is such that $\langle w, w\rangle=0$, its orthogonal projections $w_1=R_2e_1$ and $w_2=-R_1f_1$ onto 
$\mathbb{R}^{k+1, 1}$ and $\mathbb{R}^{n-k+1, 1}$, respectively, satisfy $\langle w_1, w_1\rangle=0=\langle w_2, w_2\rangle$,  and $\langle f(x), w\rangle=0$ for all $x=(s,x_1, x_2)\in M^n$. 

To recover the presentation of 
Example \ref{ex:int2}\ref{it:ex43} from \eqref{casec},
let $\bar{w}=\frac{2}{R^2}(-R_2e_2+R_1f_2)$ (where $R^2 = R_1^2+R_2^2$), which satisfies $\langle \bar w, \bar w\rangle=0$ and $\langle  w, \bar w\rangle=1$. 
Define
\[v(x)=\langle f(x), \bar w\rangle= \frac{R_1R_2}{R^2}(\|x_1\|^2-\|x_2\|^2)(\cosh \theta(s) +\sinh \theta(s))\] for every $x=(s,x_1, x_2)\in M^n$, and $\iota\colon M^n\to \mathbb{R}^{n+3, 2}$ by $\iota(x)=f(x)-v(x)w$. Then 
$\iota(M^n)\subset \mathbb{R}^{n+1,1}=\Span\,(w, \bar w)^\perp$ and  $\langle \iota(x), \iota(x)\rangle=
\langle f(x), f(x)\rangle=\kappa_1^{-1}+ \kappa_2^{-1}=-R^{2}$ for every $x\in M^n$.
We write $\kappa = -R^{-2}$.
Therefore, endowing $M^n$ with the metric induced by $\iota$, the latter gives rise to an (umbilical) isometric immersion of $M^n(\kappa)$ into $\mathbb{R}^{n+1,1}$  such that $f(x)=\iota(x)+v(x)w$
for every $x\in M^n$. 

\section{A rescaling procedure}\label{rescaling}

In this section, by means of a rescaling procedure,  we show that any hypersurface $f\colon M^3(\varepsilon/2)\to\tM^2(\varepsilon)\times\tM^{2}(\varepsilon)$, $\varepsilon\in \{-1, 1\}$, with  $\rk S = 2$ and everywhere vanishing product angle function $C$,  gives rise to a hypersurface $f'\colon M^3(\kappa)\to\tM^2(\kappa_1)\times\tM^{2}(\kappa_2)$, 
  $\kappa = \frac{\kappa_1\kappa_2}{\kappa_1+\kappa_2}$, whose product angle function has the constant value $C=\frac{\kappa_1-\kappa_2}{\kappa_1+\kappa_2}$.

In what follows, we use subscripts to keep track of the corresponding factor. 
Consider the product $\tM_1^2(\varepsilon)\times\tM_2^{2}(\varepsilon)$, $\varepsilon\in \{-1, 1\}$,  lying in the product $\br^{3, \varepsilon}_1 \times \br^{3, \varepsilon}_2$ of two mutually orthogonal Euclidean or Lorentzian spaces, according as $\varepsilon=1$ or $\varepsilon=-1$, respectively. For $i =1,2$, denote $\ip_i$ the inner product on $\br^{3, \varepsilon}_i$, and $\pi_i\colon \br^{3, \varepsilon}_1 \times \br^{3, \varepsilon}_2 \to \br^{3, \varepsilon}_i$ the corresponding orthogonal projection (and its derivative $\pi_i^*$). Note that the projections $\pi_i$ do not depend on the inner products $\ip_1$ and $\ip_2$. We denote $\ip$ the inner product on $\br^{3, \varepsilon}_1 \times \br^{3, \varepsilon}_2$ defined by $\<X,Y\>= \<\pi_1 X, \pi_1 Y\>_1 + \<\pi_2 X, \pi_2 Y\>_2$, for $X, Y\in \br^{3, \varepsilon}_1 \times \br^{3, \varepsilon}_2$.

For a hypersurface $f\colon M^3(\varepsilon/2)\to \tM_1^2(\varepsilon)\times\tM_2^{2}(\varepsilon)$, $\varepsilon\in \{-1, 1\}$, with  $\rk S = 2$, the following properties are known from \cites{LVWY,LVWYhyp}: 

\begin{itemize}[leftmargin=1.5em]
  \item \label{it:frame'}
  The product angle function $C$ is everywhere vanishing. Equivalently, the eigenfunction $r$ of $\cR$ taking values in $(0,1)$ has the constant value $r=1/2$. Therefore, there exists an orthonormal frame $(E_1, E_2, E_3)$ on $M^3$ such that $\cR E_1=0$, $\cR E_2=E_2$ and 
  $\cR E_3=\frac{1}{2}E_3$.

  \item
  The integral curves of $E_3$ are geodesics in $M^3(\varepsilon/2)$ and are mapped by $f$ into geodesics of $\tM_1^2(\varepsilon)\times\tM_2^{2}(\varepsilon)$.
  
  \item
  The components $S_{ij}:=\<\nabla^0_{E_i}f_*E_j, \xi\>$ of the shape operator $S$ with respect to a unit normal vector field $\xi$ (where $\nabla^0$ is the flat Levi-Civita connection of $\ip$ on $\br^{3, \varepsilon}_1 \times \br^{3, \varepsilon}_2$) satisfy
  \begin{equation} \label{eq:Aij'}
    S_{13} = S_{23} = S_{33} = 0, \qquad  S_{11} S_{22} - S_{12}^2 = \frac{\varepsilon}{2}.
  \end{equation}
\end{itemize}

\smallskip
For $\kappa_i=\varepsilon/R_i^2$, $i=1, 2$, set $R^2=R_1^2+R_2^2$. Define the new inner products $\ip_i'$ on $\br^{3, \varepsilon}_i$ by scaling the existing ones $\ip_i$ as follows: $\ip'_i = R_i^{2} \ip_i$, for $i =1,2$, and define $\ip'=\ip'_1 + \ip'_2$. Equivalently, for $X, Y \in \br^{3, \varepsilon}_1 \times \br^{3, \varepsilon}_2$, 
\begin{equation} \label{eq:ip'}
  \<X, Y\>' = R_1^{2} \<\pi_1 X, \pi_1 Y\>_1 + R_2^{2}\<\pi_2 X, \pi_2 Y\>_2.
\end{equation}
Notice that $\pi_1$ and $\pi_2$ do not depend on how we define the inner product on the factors. We also notice that, although the identity mapping is clearly not an isometry between $\ip'$ and $\ip$ (unless $R_1=R_2 = 1$), it is \emph{affine} relative to the Levi-Civita connections for $\ip$ and $\ip'$. In other words, the Levi-Civita connection $(\nabla^0)'$ of $\ip'$ in $\br^{3, \varepsilon}_1 \times \br^{3, \varepsilon}_2$ is the same as $\nabla^0$, so $(\nabla^0)'_XY = \nabla^0_XY$, for any two vector fields $X$ and $Y$ in $\br^{3, \varepsilon}_1 \times \br^{3, \varepsilon}_2$.

Relative to the inner product $\ip'$, $\tM_i^2(\varepsilon)$  becomes $\tM_i^2(\kappa_i)$, and so $f$ becomes an immersion $f'\colon M^3\to \tM_1^2(\kappa_1)\times\tM_2^{2}(\kappa_2)\subset (\br^{3, \varepsilon}_1 \times \br^{3, \varepsilon}_2, \ip')$.

\begin{proposition} \label{prop:scale}
  The metric induced by $f'$ on $M^3$ from $\ip'$  has constant sectional curvature $\kappa=\frac{\kappa_1\kappa_2}{\kappa_1+\kappa_2}$, and the shape operator of $f'$ also has rank $2$.
\end{proposition}

\begin{proof}
Let $(E_1', E_2', E_3')$ be the frame on $M^3$ defined by  \begin{equation}\label{eq:newbasis}
        E_1' = R_1^{-1} E_1, \quad E_2' = R_2^{-1} E_2, \quad E_3' =\sqrt{2} R^{-1} \, E_3.
        \end{equation}
        Since $t=1-r=1/2$, the vector field  $\zeta$ defined by \eqref{zeta} is $\zeta=\frac{1}{2}E_3$ by the last two equations in \eqref{propR}. Therefore, \eqref{R} and \eqref{R2} give 
        $\pi_i^*f'_*E'_j=\delta_{ij}R_j^{-1}f'_*E_j$
 for $1\leq i,j\leq 2$, and 
$\pi_j^*f'_*E_3'=\frac{\sqrt{2}}{2}R^{-1}(f'_*E_3+\xi)$ for $j=1,2$. Since $R^2=R_1^2+R_2^2$, it follows from \eqref{eq:ip'} that $(E_1', E_2', E_3')$ is an orthonormal frame with respect to $\ip'$.  Moreover, using that
$\pi_j^*\xi=\frac{1}{2}((-1)^{j}f'_*E_3+\xi)$ for $j=1,2$ by \eqref{zeta} and \eqref{zeta2}, it follows from \eqref{eq:ip'} that
\begin{equation}\label{N}
        \xi' = \sqrt{2} R^{-1} (R_1^{-1}R_2 \pi_1^* \xi + R_2^{-1}R_1 \pi_2^* \xi)
  \end{equation}
  is a unit normal vector field along $f'$.
  Furthermore, the integral curves of the vector field $E_3'$ are the same as the integral curves of the vector field $E_3$, and as such, are geodesics of $M^3$ and are mapped by $f'$ into geodesics of $\tM_1^2(\kappa_1)\times\tM_2^{2}(\kappa_2)$ (since a geodesic in a Riemannian product remains such when we homothetically scale the metrics on the factors).

  The components of the shape operator $S'$ of $f'$ are given by $S'_{ij} = \<\nabla^0_{E_i'} f'_*E_j', \xi'\>'$ for $i,j=1,2,3$ (recall that the Levi-Civita connections of $\ip$ and of $\ip'$ coincide). We already know that $S'_{33}=0$. We compute the remaining components using~\eqref{eq:Aij'}, \eqref{eq:ip'} and \eqref{eq:newbasis}:
  \begin{align*}
    S'_{11} & = \<\nabla^0_{E_1'} f'_*E_1', \xi'\>' = R_1^{-2} \<\nabla^0_{E_1} f'_*E_1, \xi'\>' = R_1^{-2} \<\nabla^0_{E_1} f'_*E_1, \pi_1^* \xi'\>' = \<\nabla^0_{E_1} f'_*E_1, \pi_1^* \xi'\> \\
    &= \sqrt{2} R^{-1} R_1^{-1}R_2 \<\nabla^0_{E_1} f'_*E_1, \pi_1^* \xi\> = \sqrt{2} R^{-1} R_1^{-1}R_2 \<\nabla^0_{E_1} f'_*E_1, \xi\> = \sqrt{2} R^{-1}R_1^{-1}R_2 S_{11}, \\
    S'_{21} & = \<\nabla^0_{E_2'} f'_*E_1', \xi'\>' = R_1^{-1} R_2^{-1} \<\nabla^0_{E_2} f'_*E_1, \xi'\>' = R_1^{-1} R_2^{-1} \<\nabla^0_{E_2} f'_*E_1, \pi_1^* \xi'\>'\\& = R_2^{-1}R_1\<\nabla^0_{E_2} f'_*E_1, \pi_1^* \xi'\> 
    = \sqrt{2} R^{-1} \<\nabla^0_{E_2} f'_*E_1, \pi_1^* \xi\> = \sqrt{2} R^{-1} \<\nabla^0_{E_2} f'_*E_1, \xi\>\\& = \sqrt{2} R^{-1} S_{21},\\
    S'_{22} & = \<\nabla^0_{E_2'} f'_*E_2', \xi'\>' = R_2^{-2} \<\nabla^0_{E_2} f'_*E_2, \xi'\>' = R_2^{-2} \<\nabla^0_{E_2} f'_*E_2, \pi_2^* \xi'\>' = \<\nabla^0_{E_2} f'_*E_2, \pi_2^* \xi'\> \\
    &= \sqrt{2} R^{-1} R_2^{-1}R_1 \<\nabla^0_{E_2} f'_*E_2, \pi_2^* \xi\> = \sqrt{2} R^{-1} R_2^{-1}R_1 \<\nabla^0_{E_2} f'_*E_2, \xi\> = \sqrt{2} R^{-1} R_2^{-1}R_1 S_{22}, \\
    S'_{31} & = \<\nabla^0_{E_3'} f'_*E_1', \xi'\>' = \sqrt{2} R^{-1} R_1^{-1} \<\nabla^0_{E_3} f'_*E_1, \xi'\>' = \sqrt{2} R^{-1} R_1^{-1} \<\nabla^0_{E_3} f'_*E_1, \pi_1^* \xi'\>' \\
    &= \sqrt{2} R^{-1} R_1\<\nabla^0_{E_3}f'_*E_1, \pi_1^* \xi'\> = 2 R^{-2} R_2 \<\nabla^0_{E_3} f'_*E_1, \pi_1^* \xi\> = 2 R^{-2} R_1 \<\nabla^0_{E_3} f'_*E_1, \xi\> = 0,\\
    S'_{32} & = \<\nabla^0_{E_3'} f'_*E_2', \xi'\>' = \sqrt{2} R^{-1} R_2^{-1} \<\nabla^0_{E_3} f'_*E_2, \xi'\>' =\sqrt{2} R^{-1} R_2^{-1} \<\nabla^0_{E_3} f'_*E_2, \pi_2^* \xi'\>' \\
    &= \sqrt{2} R^{-1} R_2\<\nabla^0_{E_3} f'_*E_2, \pi_2^* \xi'\> = 2 R^{-2} R_1 \<\nabla^0_{E_3} f'_*E_2, \pi_1^* \xi\> = 2 R^{-2} R_1\<\nabla^0_{E_3} f'_*E_2, \xi\> = 0.
  \end{align*}

  It follows that $E_3'$ lies in the kernel of $S'$, and $S'_{11} S'_{22} - (S'_{12})^2 =  2 R^{-2} (S_{11} S_{22} - S_{12}^2)=\varepsilon R^{-2}=\kappa$, by~\eqref{eq:Aij'}.  Hence $\rk S' = 2$. Moreover, from Gauss equation, the only nonzero components of the curvature tensor $R'$ of $M^3$, up to $\mathbb{Z}_2$ symmetries, are
  \begin{align*}
  R'(E_1',E_2',E_2',E_1') & = \kappa, \\
  R'(E_1',E_3',E_3',E_1') & = \kappa_1\<\pi_1^*f'_*E_3',\pi_1^*f'_*E_3'\>' = 2\kappa_1R_1^2R^{-2} \<\pi_1^*f'_*E_3,\pi_1^*f'_*E_3\> = \varepsilon R^{-2}=\kappa,\\
  R'(E_2',E_3',E_3',E_2') & = \kappa_2\<\pi_2^*f'_*E_3',\pi_2^*f'_*E_3'\>' = 2\kappa_2R_2^2R^{-2}  \<\pi_2^*f'_*E_3,\pi_2^*f'_*E_3\> = \varepsilon R^{-2}=\kappa,
  \end{align*}
  by \eqref{eq:ip'} and \eqref{eq:newbasis}. This completes the proof.
  \end{proof}

\vspace{1ex}
The next example provides an explicit parametrisation of the hypersurface
$f\colon M^3(\kappa)\to\tM^2(\kappa_1)\times\tM^{2}(\kappa_2)$, $\kappa_1\kappa_2>0$, $\kappa = \frac{\kappa_1\kappa_2}{\kappa_1+\kappa_2}$, with $\rk S =2$ and constant product angle function $C=\frac{\kappa_1-\kappa_2}{\kappa_1+\kappa_2}$, that one obtains by applying the preceding rescaling procedure to \cite[Example~3.6]{LVWYhyp}. 

  \begin{example}\label{ex:3.6} Let  $\gamma_1\colon I_1\subset \mathbb{R}\to \tM^{2}(\kappa_1)$ and $\gamma_2\colon I_2\subset \mathbb{R}\to \tM^{2}(\kappa_2)$, where $I_1$ and $I_2$ are open intervals and $\kappa_1\kappa_2>0$,  be smooth unit-speed curves with curvatures $\lambda_1$ and $\lambda_2$, respectively, let $N_1$ and $N_2$ be unit normal vector fields along $\gamma_1$ and $\gamma_2$, respectively, and let $\theta\in \mathbb{R}$.  Define $f\colon \mathbb{R}\times I_1\times I_2\to \tM^2(\kappa_1)\times\tM^{2}(\kappa_2)$  by 
 \begin{equation}\label{param}
f(s, u_1, u_2)=(f_1(s, u_1), f_2(s, u_2)),
\end{equation}
where
 \begin{equation}\label{f1}
 f_1(s, u_1)=C_{\varepsilon}(s \cos \theta  / R_1) \gamma_1(u_1) + R_1 S_{\varepsilon}(s\cos \theta  / R_1)N_1(u_1)
\end{equation}
and
 \begin{equation}\label{f2}
f_2(s, u_2)=C_{\varepsilon}( s\sin \theta / R_2) \gamma_2(u_2) + R_2 S_{\varepsilon}(s\sin \theta  / R_2)N_2(u_2).
\end{equation}
Here $\varepsilon=\kappa_1/|\kappa_1|=\kappa_2/|\kappa_2|$, $\kappa_i=\varepsilon/R_i^2$, and $C_\varepsilon$ and $S_\varepsilon$ are as in \eqref{cs}. 
Restricted to the open subset of its regular points, the map $f$ defines a hypersurface having the coordinate vector fields as principal directions, with $0$ as one of its principal curvatures, the $s$-coordinate curves being its associated  curvature lines, and two nonzero principal curvatures
\begin{equation}\label{eq:pc}
\frac{- \sin\theta }{R_1}\frac{\varepsilon S_\varepsilon(s \cos \theta/R_1)+\lambda_1 R_1 C_\varepsilon(s \cos\theta/R_1)}{C_\varepsilon(s \cos\theta/R_1)-\lambda_1 R_1  S_\varepsilon(s \cos\theta/R_1)}, \,\, \frac{\cos \theta}{R_2}\frac{\varepsilon S_\varepsilon(s \sin \theta/R_2)+\lambda_2 R_2 C_\varepsilon(s \sin\theta/R_2)}{C_\varepsilon(s \sin\theta/R_2)-\lambda_2 R_2  S_\varepsilon(s \sin \theta/R_2)}.
\end{equation}
The coordinate vector fields are also principal directions of the tensor $\cR$ defined in \eqref{R}, with eigenvalues
$0, 1$ and $\sin^2 \theta$, the latter corresponding to the coordinate vector field $\partial/\partial s$. In particular, 
the tensors $S$ and $\cR$ satisfy $S\cR=\cR S$.
Taking $\theta=\arcsin{\left(\frac{\kappa_1}{\kappa_1+\kappa_2}\right)^{1/2}}$, the eigenfunction of $\cR$ taking values in $(0,1)$ has the constant value $r=\sin^2\theta=\frac{\kappa_1}{\kappa_1+\kappa_2}$, hence
 the product angle function has the constant value $C=2r-1=\frac{\kappa_1-\kappa_2}{\kappa_1+\kappa_2}$.
 Moreover, choosing the curves $\gamma_1$ and $\gamma_2$ to have constant curvatures $\lambda_1$ and $\lambda_2$, respectively,  with $\lambda_1\lambda_2=(\kappa_1\kappa_2)^{1/2}$,
one can check using Gauss equations that $f$ has constant sectional curvature  $\kappa = \frac{\kappa_1\kappa_2}{\kappa_1+\kappa_2}$. 
\end{example}

\begin{remark} Hypersurfaces of $\tM^2(\varepsilon)\times\tM^{2}(\varepsilon)$, $\varepsilon\in \{-1, 1\}$, with constant sectional curvature and $\rk S=2$ satisfying  $S\cR=\cR S$, correspond to \cite[Theorem~1.1(1)]{LVWY} for $\varepsilon=1$, and to \cite[Theorem~1.1(2)]{LVWYhyp} for $\varepsilon=-1$, in the classification of all hypersurfaces with constant sectional curvature of 
$\tM^2(\varepsilon)\times\tM^{2}(\varepsilon)$.  
The case of hypersurfaces with constant sectional curvature and $\rk S =2$ for which  $S\cR\neq  \cR S$ correspond to \cite[Theorem~1.1(2)]{LVWY} for $\varepsilon=1$ and to \cite[Theorem~1.1(3)]{LVWYhyp} for $\varepsilon=-1$. More precisely, it was shown in \cite{LVWY} that any hypersurface
$f\colon M^3(1/2)\to\mathbb{S}^2\times \mathbb{S}^2$ with $\rk S = 2$ for which $S\cR\neq \cR S$ is a parallel hypersurface at a distance $\tfrac{\pi}{2\sqrt{2}}$ of a minimal hypersurface of $\mathbb{S}^2\times \mathbb{S}^2$ with vanishing product angle function, and such minimal hypersurfaces were described in Theorem 3.5 of that paper. On the other hand, hypersurfaces
$f\colon M^3(-1/2)\to\mathbb{H}^2\times \mathbb{H}^2$ with $\rk S = 2$ for which $S\cR\neq \cR S$ were described in \cite[Theorems 5.7--5.10]{LVWYhyp}. In either case, $f$ has a vanishing product angle function.
\end{remark}

\section{The classification}\label{ss:classification}

We are now in a position to state and prove the following more precise version of the classification of all hypersurfaces with constant sectional curvature and dimension $n\geq 3$ of a product of two space forms.

\begin{theorem}\label{mainthmb}  
Let $f\colon M^n\to \tM^{n+1}$, $n\geq 3$, be a {\rm(}smooth{\rm)} isometric immersion of a space of constant sectional curvature $M^n = M^n(\kappa)$ in a product $\tM^{n+1}=\tM^{k}(\kappa_1)\times \tM^{n-k+1}(\kappa_2)$ of two space forms.

\begin{enumerate}[label={\rm{(\roman*)}}]
\item\label{it:thm11} If $\tM^{n+1}=\mathbb{R}\times \tM^{n}(\kappa_2)$, $\kappa_2\neq 0$, then $f$ is as in \cite[Theorems~7.1 and 7.2]{MaT}{\rm;}
\item\label{it:thm12} If $\tM^{n+1}=\mathbb{R}^k\times \tM^{n-k+1}(\kappa_2)$, $2\leq k\leq n-1$, $\kappa_2\neq 0$,  then $f$ is as in  {\cite[Theorems~11 and 12]{CT}}{\rm;}
 \item If $\kappa_1\kappa_2\neq 0$, then $\kappa = \frac{\kappa_1\kappa_2}{\kappa_1+\kappa_2}$ and one of the following holds{\rm:}
        \begin{enumerate}[label={\rm{(\roman{enumi}.\alph*)}}]
            \item\label{it:thm13} Either  $\kappa_1\kappa_2<0$, with $\kappa_1+\kappa_2<0$, or $\kappa_1,\kappa_2<0$, and there exists an open subset $\mathcal{U}\subset \tM^n(\kappa)$ and an isometry $\phi\colon \mathcal{U}\to M^n$ such that $f\circ \phi$ is the restriction to $\mathcal{U}$ of an isometric embedding of $\tM^{n}(\kappa)$ into $ \tM^{n+1} \subset \tM^{n+2,1}(\kappa)$ as in Examples~\ref{ex:int1} or \ref{ex:int2}, respectively{\rm;}
            \item\label{it:thm14} $n=3$, $k=2$,   $\kappa_1\kappa_2>0$, and  $f$ is constructed,  by means of the rescailing procedure in Proposition \ref{prop:scale}, from the examples in  \cites{LVWY,LVWYhyp} with  $\rk S=2$ for the case $\kappa_1=\kappa_2=\pm1$.
        \end{enumerate}
\end{enumerate}
\end{theorem}

\proof
We use the notation of Section~\ref{sec:hypprod}.
We first dispose of the case $\rk S\le 1$, which is ruled out by Lemma~\ref{lem:rks01}.
The cases $\kappa_1\kappa_2=0$ are already covered in \cite{MaT} and \cite{CT}, as stated in 
items~\ref{it:thm11} and~\ref{it:thm12}. 

Henceforth, we assume that $\kappa_1\kappa_2\ne0$ and $2\le k\le n-1$. 

 Let $\xi$ be a unit normal vector field along $f$, and let
    $
    h\colon \tM^{k}(\kappa_1)\times\tM^{n-k+1}(\kappa_2) \to \br^{k+1,\sigma(\kappa_1)}\times \br^{n-k+2,\sigma(\kappa_2)} = \br^{n+3,\eta}
    $
    be the inclusion, where $\eta = \sigma(\kappa_1)+\sigma(\kappa_2)$. Then
    \[
    h(\tM^{k}(\kappa_1)\times\tM^{n-k+1}(\kappa_2)) \subset \tM^{n+2,\rho}(\tilde{\kappa}),\quad \tilde{\kappa} = \frac{\kappa_1\kappa_2}{\kappa_1+\kappa_2},
    \]
    where
   $$
 \rho=\eta - \sigma(\tilde{\kappa})=\left\{ \begin{array}{l} 0, \,\,\,\mbox{if}\,\,\, \kappa_1, \kappa_2>0,\,\,\,\mbox{or} \,\,\, \kappa_1\kappa_2<0\,\,\,\mbox{and}\,\,\,
 \kappa_1+\kappa_2>0,\\
 1,\,\,\,\mbox{if}\,\,\, \kappa_1, \kappa_2<0,\,\,\,\mbox{or} \,\,\,\kappa_1 \kappa_2<0\,\,\,\mbox{and}\,\,\, \kappa_1+\kappa_2<0.\\
 \end{array}\right.
 $$
 Denote $i\colon\tM^{k}(\kappa_1)\times\tM^{n-k+1}(\kappa_2) \to \tM^{n+2,\rho}(\tilde{\kappa})\subset \br^{n+3,\eta}$ the inclusion, and let $\tilde{f} = i\circ f$. Define 
 \begin{equation}\label{eq:UV}
     U =-|\tilde \kappa|^{1/2}(R_1\nu_1+R_2\nu_2)\,\,\,\,\mbox{and}\,\,\,\,V=|\tilde \kappa|^{1/2}(-\varepsilon_1R_2\nu_1+\varepsilon_2R_1\nu_2).
 \end{equation}
It follows from \eqref{eq:nu} and \eqref{eq:nu'} that 
$$S_U=|\tilde \kappa|^{1/2}\,I\,\,\,\,\mbox{and}\,\,\,\,S_V=aI+b\cR,$$
  with
 \begin{equation}\label{eq:ab} a=|\tilde\kappa|^{1/2}\frac{R_2}{\varepsilon_1R_1}=\frac{\kappa_1}{|\kappa_1+\kappa_2|^{1/2}}\,\,\,\,\mbox{and}\,\,\,\,b=-|\tilde\kappa|^{1/2}\frac{\varepsilon_2R_1^2+\varepsilon_1R_2^2}{R_1R_2}=
 -\frac{\kappa_1+\kappa_2}{|\kappa_1+\kappa_2|^{1/2}}.
 \end{equation}
 Thus, $U$ is normal to $\tM^{n+2,\rho}(\tilde\kappa)$ along $\tilde{F}=h\circ f$, and $\{i_*\xi,V\}$ is an orthonormal frame of the normal space $N_{\tilde{f}}M^n$.
 
We start with the following fact.

\begin{lemma}\label{lem:setcons}
    In the assumptions of Theorem~\ref{mainthmb}, with $\kappa_1\kappa_2\ne0$ and $2\le k\le n-1$, the following holds.
    \begin{enumerate}[label={\rm{(\alph*)}}]
        \item\label{it:setcons1} If $n\ge 4$ or $n=\rk S=3$, then  $\kappa_1 \kappa_2<0$, with $\kappa_1+\kappa_2<0$, or $\kappa_1, \kappa_2<0$. 
        \item\label{it:setcons2} If $n=3$ and $\rk S=2$, then $\kappa_1\kappa_2>0$.
    \end{enumerate}
    Moreover, in both cases, $\kappa=\frac{\kappa_1\kappa_2}{\kappa_1+\kappa_2}$.
\end{lemma}
\begin{proof}
   For assertion~\ref{it:setcons1}, we first suppose that $n\ge 4$ and $\kappa\ne\tilde\kappa$. Let $j\colon\tM^{n+2,\rho}(\tilde\kappa)\to \tM^{n+3,\delta}(\kappa)$ be an umbilical isometric immersion, where  $\delta=\rho$ if $\tilde\kappa > \kappa$ and  $\delta=\rho+1$ if $\tilde\kappa < \kappa$. Let $F=j\circ \tilde f$. Then the second fundamental form $\secf_{F}(x)$ of $F$ at any $x\in M^n$ is flat with respect to the inner product in the normal space $N_{F}M^n(x)$, which has index
 $$
\left\{ \begin{array}{l} 0, \,\,\,\mbox{if either}\,\,\, \kappa_1, \kappa_2>0,\,\,\mbox{or}
 \,\,\, \kappa_1\kappa_2<0,\,\kappa_1+\kappa_2>0\,\,\,\mbox{and}\,\,\,\tilde\kappa > \kappa,\\
 2,  \,\,\,\mbox{if either}\,\,\,\kappa_1, \kappa_2<0\,\,\mbox{and}\,\, \tilde\kappa < \kappa,\,\,\mbox{or}
 \,\,\, \kappa_1\kappa_2<0,\, \kappa_1+\kappa_2<0\,\,\,\mbox{and} \,\,\tilde\kappa < \kappa,\\
 1, \,\,\,\mbox{in the remaining cases}.
 \end{array}\right.$$
  Notice that the inner product in the (one-dimensional) normal space to $j$ at $x$ is positive definite if $\tilde\kappa > \kappa$ and negative definite otherwise. Let $\nu$ be one of the two normal vectors to $j$ at $x$ such that $\<\nu, \nu\>=(\tilde\kappa - \kappa)/|\tilde\kappa - \kappa|$ and 
  $S_\nu= \sqrt{|\tilde\kappa - \kappa|}\,I$. In particular, the latter equation implies that $\secf_F(x)$ has a trivial kernel.    
  
 The case in which the inner product on $N_F M^n(x)$ is positive definite is then ruled out by Lemma~\ref{lem:flatbi}\eqref{it:pode},
 which implies that $0=\dim \mathcal{N}(\secf_{F}(x))\geq n-3$, in contradiction with the assumption that $n\geq 4$. 
 Multiplying by $-1$ the inner product on $N_FM^n(x)$ if it has index~$2$, it becomes a Lorentzian inner product. 
  Therefore, in the cases in which the inner product on $N_FM^n(x)$ has index either $1$ or $2$, Lemma~\ref{lem:flatbi}\eqref{it:symim} implies that the subspace $\mathcal{I}(\secf_F(x))$ spanned by the image of $\secf_F(x)$ must be degenerate for every $x\in M^n$, for otherwise we would reach the same contradiction as in the positive definite case. 

  Thus, there exists a nonzero vector $N\in \mathcal{I}(\secf_F(x))\cap (\mathcal{I}(\secf_F(x)))^\perp$. Observe that $N$ cannot be orthogonal to $\nu$, as otherwise $\mathcal{I}(\secf_{\tilde f}(x))$ would be a degenerate subspace of a Lorentzian plane, and hence would be spanned by a light-like vector. This would imply that $\kappa =\tilde\kappa$ from Gauss equation of $\tilde f$, a contradiction. Consequently, we may assume that $N=i_*N_1+\nu$ for some $N_1 \in N_{\tilde{f}}M^n(x)$ with $\<N_1,N_1\> = \varepsilon = -(\tilde\kappa - \kappa)/|\tilde\kappa - \kappa|$. Since
    $$0=\<\secf_{F}(X, Y), N\>=\<i_*\secf_{\tilde f}(X, Y)+ \sqrt{|\tilde\kappa-\kappa|}\<X, Y\>\nu, i_*N_1+\nu\>,$$
    it follows that $S_{N_1} = \varepsilon\sqrt{|\tilde\kappa - \kappa|}I$.
    Choosing $N_2 \in N_{\tilde{f}}M^n(x)$ orthogonal to $N_1$ with $\<N_2,N_2\> = -\varepsilon$, Gauss equation of $\tilde{f}$ implies that $\rk S_{N_2} \le 1$. Therefore, $\{N_1,N_2\}$ is an orthonormal basis of $N_{\tilde{f}}M^n(x)$ such that
     $$S_{N_1}=\varepsilon \sqrt{|\tilde\kappa - \kappa|}\,I\,\,\,\mbox{and}\,\,\rk S_{N_2}\le 1.$$
    Now let $\theta \in C^\infty(M^n)$ be such that 
     $$i_*\xi=\cos \theta N_1+\sin \theta N_2 \,\,\,\,\mbox{and}\,\,\,\,V=-\sin\theta N_1+\cos\theta N_2.$$ 
     Then, on $\ker S_{N_2}$ we have
      $$S=\varepsilon\sqrt{|\tilde\kappa - \kappa|}\cos\theta\,I\,\,\, \mbox{and}\,\,\, aI+b\cR=-\varepsilon\sqrt{|\tilde\kappa - \kappa|}\sin\theta\,I.$$
      Since $\dim\ker S_{N_2} \ge n-1$, it follows from the second of the above equations that $\cR$ can have at most two distinct eigenvalues. As $\dim\ker\cR \ge k-1 \ge 1$ and $\dim\ker(I-\cR)\ge n-k \ge 1$ by \eqref{dimkerR} and our assumption that $2\leq k\leq n-1$, such eigenvalues must be $0$ and $1$, with multiplicities  $(n-1, 1)$ or $(1, n-1)$. It follows from \eqref{dimkerR} that $k=n-1$ in the former case and $k=2$ in the latter. Moreover, $f$ is locally decomposable by Proposition~\ref{splits}; in particular, $\kappa=0$. But since both factors have dimensions greater than one and one of them is an inclusion,
     this leads to a contradiction.

     Now we have $\kappa=\tilde\kappa=\frac{\kappa_1\kappa_2}{\kappa_1+\kappa_2}$. Then, at any $x\in M^n$, the second fundamental form $\secf_{\tilde{f}}(x)$ is a flat bilinear form with respect to the inner product on $N_{\tilde{f}}M^n(x)$, whose index is equal to $\rho$. Therefore, if $\rho=0$, then $\dim\mathcal{N}(\secf_{\tilde f}(x)) \ge n-2$ by Lemma~\ref{lem:flatbi}\eqref{it:pode}.
     On the subspace $\mathcal{N}(\secf_{\tilde f})$ at $x\in M^n$, we have $S=0=aI+b\cR$. Since $-ab^{-1}=\kappa_1(\kappa_1+ \kappa_2)^{-1}$, by the assumption that $n\geq 4$ this contradicts the fact that the eigenvalues of $\cR$ lie in $[0,1]$ and that the unique eigenvalue of $\cR$ in $(0,1)$ has multiplicity one.
     Thus, we must have $\rho=1$, hence  $\kappa_1 \kappa_2<0$, with $\kappa_1+\kappa_2<0$, or $\kappa_1, \kappa_2<0$.

     Next, when $n=\rk S=3$, the claim follows from \cite[Lemma~6(a)]{KNT}.

     We now prove assertion~\ref{it:setcons2}. Suppose that $n=3$ and $\rk S=2$. Arguing by contradiction, assume $\kappa\ne\tilde\kappa$. Then the same argument as in the case $n\ge 4$ applies, except that we must also consider the case in which the inner product on $N_{F}M^3(x)$ at a point $x\in M^3$ is positive definite. But then, since $\secf_{F}(x)$ has trivial kernel, $\mathcal{I}(\secf_{F}(x))=N_{F}M^3(x)$. It follows from \cite[Theorem~5.2]{DT} that there exists a frame $\{X_1,X_2,X_3\}$ that diagonalises $\secf_{F}$. As $S_{\nu}=\sqrt{|\tilde\kappa-\kappa|}I$, such a frame is orthogonal; hence, it may be chosen as an orthonormal frame. In particular, this frame diagonalises $\secf_{\tilde f}$ and so diagonalises $S$ and $aI+b\cR$, thus $S\cR = \cR S$. We take $X_2=\|\zeta\|^{-1}\zeta$, then we obtain
     \begin{gather*}
         SX_1=\lambda_1 X_1,\quad SX_2=\lambda_2 X_2,\quad SX_3=\lambda_3 X_3,\\
         \cR X_1 = 0,\quad \cR X_2 = r X_2, \quad \cR X_3 = X_3,\\
         (I-\cR)X_1 = X_1,\quad (I-\cR)X_2 = (1-r)X_2,\quad (I-\cR)X_3 = 0.
     \end{gather*}
     Gauss equation~\eqref{Gauss} gives $\kappa=\lambda_1\lambda_2 + \kappa_1(1-r)  = \lambda_2\lambda_3+\kappa_2 r = \lambda_1\lambda_3$. This yields $\lambda_1\lambda_3\ne 0$ and so $\lambda_2=0$ as $\rk S =2$. Moreover, we have $\kappa_1(1-r)=\kappa=\kappa_2 r$, which implies $\kappa=\frac{\kappa_1\kappa_2}{\kappa_1+\kappa_2}=\tilde\kappa$, a contradiction. 
     
     Now we have $\kappa=\frac{\kappa_1\kappa_2}{\kappa_1+\kappa_2}$. Then, $\secf_{\tilde{f}}(x)$ is a flat bilinear form for every $x\in M^3$, and hence it has nontrivial kernel by Lemma~\ref{lem:flatbi}\eqref{it:pode}. In particular, $S_{V} = aI + b\cR$ has nontrivial kernel and $r=-ab^{-1} = \kappa_1(\kappa_1+\kappa_2)^{-1}$. Since $\mathcal{N}(\secf_{\tilde f}) = \ker S \cap \ker S_V$, a vector field $Y$ spanning $\mathcal{N}(\secf_{\tilde f})$ satisfies $SY= 0 $ and $\cR Y = rY$. Computing the sectional curvatures of the planes spanned by $\{X,Y\}$ and $\{Y,Z\}$, where $\{X,Y,Z\}$ is an orthonormal frame such that $\cR X = 0$ and $\cR Z= Z$, and using Gauss equation of $F$, we have $\kappa_1(1-r)=\kappa=\kappa_2 r$, which implies $\kappa_1\kappa_2>0$.
\end{proof}

The proof of the theorem is divided into three cases: $n\ge 4$, $n=\rk S=3$, and $n=3$ with $\rk S=2$. 

\vspace{1ex}

\underline{Case 1. $n\ge 4$}. By Lemma~\ref{lem:setcons}\ref{it:setcons1}, we have $\kappa=\frac{\kappa_1\kappa_2}{\kappa_1+\kappa_2}$, with $\kappa_1+\kappa_2<0$.

     If $\mathcal{I}(\secf_{\tilde f}(x))$ was nondegenerate, we would reach the contradiction using Lemma~\ref{lem:flatbi}\eqref{it:symim}. 
     As a result, $\mathcal{I}(\secf_{\tilde f}(x))$ is degenerate for any $x\in M^n$.  This means that, for all $x\in M^n$ and all $X, Y\in T_xM^n$,
 $$0=\<\secf_{\tilde f}(X,Y), i_*\xi+V\>=\<SX, Y\>-\<(aI+b\cR)X, Y\>,$$
 hence  $S=aI+b\cR$ (and in particular $S\cR = \cR S$). 

 Let us first assume that $\kappa_1\kappa_2<0$ and $\kappa_1+\kappa_2<0$, and suppose $\kappa_2<0<\kappa_1$. Then the second fundamental form of $\tilde{F}=h\circ f$ is
  \begin{eqnarray*}
  \secf_{\tilde F}(X,Y)&=&h_*\secf_f(X,Y) -\frac{1}{R_1}\<(I-\cR)X, Y\> \nu_1+\frac{1}{R_2}\<\cR X, Y\>
    \nu_2\\
    &=& \<(aI+b\cR)X, Y\>h_*\xi -\frac{1}{R_1}\<(I-\cR)X, Y\> \nu_1+\frac{1}{R_2}\<\cR X, Y\>\nu_2\\
    &=& \<X, Y\>\left(a h_*\xi-\frac{1}{R_1}\nu_1\right)+\<\cR X,Y\>\left(b h_*\xi+\frac{1}{R_1}\nu_1+\frac{1}{R_2}\nu_2\right),
    \end{eqnarray*}
  where $\kappa_1=R_1^{-2}$ and $\kappa_2=-R_2^{-2}$.
  Thus the first normal space $N^1_{\tilde F}M^n(x)$ of $\tilde{F}$ at any $x\in M^n$ is spanned by the vectors $ah_*\xi-\frac{1}{R_1}\nu_1$ and $bh_*\xi+\frac{1}{R_1}\nu_1+\frac{1}{R_2}\nu_2$, or equivalently, using~\eqref{eq:ab},
   by the vectors 
   $$Z_1=\kappa^{1/2}R_2 h_*\xi -\nu_1\,\,\,\mbox{and}\,\,\,Z_2=\kappa^{-1/2}h_*\xi + R_2\nu_1+R_1\nu_2.$$
   Notice that $Z_2\in (\Span(Z_1, Z_2))^\perp,$ hence $\Span(Z_1, Z_2)$ is a degenerate vector subbundle of rank two. Now we have, for any $X\in\Gamma(TM^n)$,
 \begin{eqnarray*}
 \nabla_X^\perp h_*\xi&=&\secf_h(f_*X,h_*\xi)\\
 &=&-\frac{1}{R_1}\<\tilde \pi_1 \xi,f_*X\> \nu_1+\frac{1}{R_2}\<\tilde \pi_2 \xi,f_*X\>\nu_2\\
 &=&\<X, \zeta\>\left(\frac{1}{R_1}\nu_1+\frac{1}{R_2}\nu_2\right),
 \end{eqnarray*}
 $$\nabla_X^\perp\nu_1=-\frac{1}{R_1}\<X,\zeta\>h_*\xi\,\,\,\,\mbox{and}\,\,\,\,\nabla_X^\perp\nu_2=\frac{1}{R_2}\<X,\zeta\>h_*\xi.$$
 Thus,
  $$
 \begin{array}{l}
 \nabla_X^\perp Z_2=\kappa^{-1/2}\<X, \zeta\>\left(\frac{1}{R_1}\nu_1+\frac{1}{R_2}\nu_2\right)-\frac{R_2}{R_1}\<X,\zeta\>h_*\xi+\frac{R_1}{R_2}\<X,\zeta\>h_*\xi\vspace{2ex}\\
 \hspace*{7ex}= \frac{\kappa^{-1/2}}{R_1R_2}\<X,\zeta\> Z_2.
 \end{array}
 $$
  Observe that, by Lemma~\ref{lem:grad}, the vector field $\zeta$ is a gradient vector field. 
  Taking (locally) a positive $\la\in C^{\infty}(M^n)$ such that 
 $\grad(\log\la)=-\frac{\kappa^{-1/2}}{R_1R_2} \zeta$ implies that $w=\lambda Z_2$ is a constant vector in $\br^{n+3,1}$, for $Z_2\in \Gamma((N_{\tilde F}^{1}M^n)^\perp)$ and
 $$\nabla_X^\perp \lambda Z_2=X(\lambda)Z_2+\lambda \nabla_X^\perp Z_2=\left(\<\grad \lambda, X\>+\frac{ \kappa^{-1/2}}{R_1R_2}\lambda\<X,\zeta\>\right)Z_2=0.$$
 Therefore  $\tilde F(M^n)$ is contained in the intersection of $\tM^{k}(\kappa_1)\times \tM^{n-k+1}(\kappa_2)\subset \br^{n+3,1}$ with a degenerate hyperplane $\mathcal{H}$ orthogonal to $w$. Since $w$ is orthogonal to $U$, which is a multiple of the position vector of $\tilde F$, then $\mathcal{H}$ passes through the origin. This proves the statement~\ref{it:thm13} when $\kappa_1\kappa_2<0$.  

  We now suppose that $\kappa_1, \kappa_2<0$, say, $\kappa_1=-R_1^{-2}$ and $\kappa_2=-R_2^{-2}$.
  Then an argument analogous to that used in the case $\kappa_1\kappa_2<0$ shows that the first normal space $N_{\tilde{F}}^{1}M^n(x)$ at any $x\in M^n$ is spanned by the vectors
  $$Y_1=-|\kappa|^{1/2}R_2h_*\xi+\nu_1\,\,\,\mbox{and}\,\,\,Y_2=|\kappa|^{-1/2}h_*\xi-R_2\nu_1+R_1\nu_2$$ and forms a degenerate vector subbundle of rank two, for $Y_2\in (N_{\tilde{F}}^{1}M^n(x))^\perp$. Using that $\nabla_X^{\perp}Y_2 = \frac{|\kappa|^{-1/2}}{R_1R_2}\<X,\zeta\> Y_2$, one shows as in the previous case that there exists a  positive $\mu\in C^\infty(M^n)$ such that $\nabla_X^\perp \mu Y_2=0$, which, together with the fact that $Y_2\in (N_{\tilde{F}}^{1}M^n(x))^\perp$, implies that $w=\mu Y_2$ is a constant vector  in $\mathbb{R}^{n+3,2}$. It follows that $\tilde{F}(M^n)$ is contained in the intersection of $\tM^{k}(\kappa_1)\times \tM^{n-k+1}(\kappa_2)\subset \br^{n+3,2}$ with a degenerate hyperplane $\mathcal{H}$ orthogonal to $w$. As before, the fact that $w$ is orthogonal to $U$, which is a multiple of the position vector of $\tilde F$, implies that $\mathcal{H}$ passes through the origin, thereby proving that $f(M^n)$ is as stated in~\ref{it:thm13} when $\kappa_1, \kappa_2<0$.

\vspace{1ex}

\underline{Case 2. $n=\rk S=3$}.
We have $\kappa_1+\kappa_2< 0$ and  $\kappa=\frac{\kappa_1\kappa_2}{\kappa_1+\kappa_2}$ by Lemma~\ref{lem:setcons}\ref{it:setcons1}.

At each point $x\in M^3(\kappa)$, the second fundamental form $\secf_{\tilde f}(x)$ is a flat bilinear form with respect to the Lorentzian inner product on $N_{\tilde f}M^3(x)$. Hence, if $\mathcal{I}(\secf_{\tilde f}(x))$ was nondegenerate,  by Lemma~\ref{lem:flatbi}\eqref{it:pode} 
the index of relative nullity of $\tilde f$ at $x$ would be at least~$1$, contradicting our assumption that $\rk S=3$. Therefore, $\mathcal{I}(\secf_{\tilde f}(x))$ must be degenerate for all $x\in M^3(\kappa)$. Consequently,
 $$0=\<\secf_{\tilde f}(X,Y), i_*\xi + V\>=\<SX, Y\>-\<(aI+b\cR)X, Y\>$$
 for all $X,Y\in \Ga(TM^3)$, which shows that $S=aI+b\cR$. 

 Now, in either of the possible cases $\kappa_1\kappa_2<0$ (with $\kappa_2<0<\kappa_1$), or $\kappa_1,\kappa_2<0$, one may argue as in the proof of the case $n\geq 4$ to conclude that the assertion \ref{it:thm13} in the statement holds.

\vspace{1ex}

\underline{Case 3. $n=3$ and $\rk S=2$}.
We have from Lemma~\ref{lem:setcons}\ref{it:setcons2} that $\kappa=\frac{\kappa_1\kappa_2}{\kappa_1+\kappa_2}$, with $\kappa_1\kappa_2>0$.

Let $(E_1, E_2, E_3)$ be an orthonormal diagonalising frame of the tensor $\cR$, with $\cR E_1=0$, $\cR E_2=E_2$,  and $\cR E_3=rE_3$, where $r$  takes values in $(0,1)$. We claim that $r$ has the constant value $r=\frac{\kappa_1}{\kappa_1+\kappa_2}$, and that $\ker S = \Span (E_3)$.

  Let $\iota\colon \tM^2(\kappa_1)\times\tM^2(\kappa_2)\to \tM^5(\kappa)$ be the inclusion, and set $\tilde{f}=\iota\circ f$.  Then, at each point $x\in M^3(\kappa)$, $\secf_{\tilde{f}}(x)$ is a flat bilinear form with respect to the inner product on the normal space $N_{\tilde f}M^3(x)$, which is positive definite if both $\kappa_1$  and $\kappa_2$ are positive, and Lorentzian if both are negative.   In the former case, it follows from    Lemma~\ref{lem:flatbi}\eqref{it:pode} that $\secf_{\tilde{f}}(x)$ has nontrivial kernel.   
  Let $(U, V)$ be the orthonormal frame of $N_{\tilde f}M^3$ given by \eqref{eq:UV}, with $\varepsilon_1=\varepsilon=\varepsilon_2$. Since 
  $\mathcal{N}(\secf_{\tilde{f}}(x))=\ker S\cap \ker S_V\subset \ker S$, and the latter has dimension one, it follows that 
$\mathcal{N}(\secf_{\tilde{f}}(x))=\ker S= \ker S_V$. Recall that $S_V=aI+b\cR$, with $a$ and $b$ given by \eqref{eq:ab}. From $S_VE_1=aE_1$,  $S_VE_2=(a+b)E_2$, and the symmetry of $S_V$, it follows that 
$\ker S_V=\Span(E_3)=\ker S$.
Therefore, 
$0=S_V E_3=(aI+b\cR) E_3$, hence $\cR E_3=-\frac{a}{b}E_3=\frac{\kappa_1}{\kappa_1+\kappa_2}E_3$,
  and the claim is proved in this case. If both $\kappa_1$  and $\kappa_2$ are negative
  and the subspace $\mathcal{I}(\secf_{\tilde f}(x)) $ spanned by the image of $\secf_{\tilde{f}}(x)$ is nondegenerate, the claim follows  as before using  Lemma~\ref{lem:flatbi}\eqref{it:symim}.
  On the other hand,  $\mathcal{I}(\secf_{\tilde f}(x)) $ being degenerate means that $S=S_V=aI+bR$, which can only have nontrivial kernel if $0=S_V E_3=(aI+b\cR)E_3$, hence $SE_3=0$ and $\cR E_3=\frac{\kappa_1}{\kappa_1+\kappa_2}E_3$, thus proving the claim also in this case.

  Notice that the argument in the preceding paragraph has shown that  
  $\mathcal{N}(\secf_{\tilde{f}}(x))=\Span(E_3)$. 
  This implies that the integral curves of $E_3$ are relative nullity leaves of $\tilde{f}$, hence they are geodesics of $M^3$ (this also follows from \eqref{derS2}), and are mapped by $\tilde{f}$ into geodesics of $\tM^5(\kappa)$, in particular into geodesics of $\tM^2(\kappa_1)\times \tM^2(\kappa_2)$.  
  In summary, any hypersurface $f\colon M^3(\kappa)\to \tM^2(\kappa_1)\times \tM^2(\kappa_2)$, with $\rk S=2$ everywhere, satisfies the following:

\begin{itemize}[leftmargin=1.5em]
  \item[\textbullet]
  $\kappa = \frac{\kappa_1 \kappa_2}{\kappa_1 + \kappa_2}$, or equivalently, $ R^2 = R_1^2 + R^2_2$, where $\kappa_1=\varepsilon R_1^{-2}$,  $\kappa_2=\varepsilon R_2^{-2}$ and $\kappa=\varepsilon R^{-2}$.
    \item[\textbullet] \label{it:frame}
  The eigenfunction $r$ of $\cR$ taking values in $(0,1)$ has the constant value $r=\frac{\kappa_1}{\kappa_1+\kappa_2}$. Equivalently, the product angle function $C$ has the constant value $C=\frac{\kappa_1-\kappa_2}{\kappa_1+\kappa_2}$. Therefore, there is an orthonormal frame $(E_1, E_2, E_3)$ on $M^3$ such that $\cR E_1=0$, $\cR E_2=E_2$ and 
  $\cR E_3=\frac{\kappa_1}{\kappa_1+\kappa_2}E_3$. 
  \item[\textbullet]
  The integral curves of $E_3$ are geodesics in $M^3(\kappa)$ and are mapped by $f$ into geodesics of $\tM^2(\kappa_1)\times\tM^{2}(\kappa_2)$.
  \item[\textbullet]
  The components $S_{ij}=\<\nabla^0_{E_i}f_*E_j, \xi\>$ of the shape operator $S$ with respect to a unit normal vector field $\xi$ (where $\nabla^0$ is the flat Levi-Civita connection of $\ip$ on $\br^{3, \varepsilon}_1 \times \br^{3, \varepsilon}_2$) satisfy
  \begin{equation} \label{eq:Aij}
    S_{13} = S_{23} = S_{33} = 0, \qquad  S_{11} S_{22} - S_{12}^2 = \kappa=\varepsilon R^{-2}.
  \end{equation}
\end{itemize}

Here, as in the paragraphs that precede Proposition \ref{prop:scale}, we denote each factor by the corresponding subscript, and consider the product $\tM_1^2(\kappa_1)\times\tM_2^{2}(\kappa_2)$ lying in the product $\br^{3, \varepsilon}_1 \times \br^{3, \varepsilon}_2$ of two mutually orthogonal Euclidean or Lorentzian spaces, according as $\varepsilon=1$ or $\varepsilon=-1$, respectively. For $i =1,2$,  $\ip_i$ denotes the inner product on $\br^{3, \varepsilon}_i$ and $\pi_i\colon \br^{3, \varepsilon}_1 \times \br^{3, \varepsilon}_2 \to \br^{3, \varepsilon}_i$ the corresponding orthogonal projection. As before, we again denote $\ip$ the inner product on $\br^{3, \varepsilon}_1 \times \br^{3, \varepsilon}_2$ defined by $\<X,Y\>= \<\pi_1 X, \pi_1 Y\>_1 + \<\pi_2 X, \pi_2 Y\>_2$, for $X, Y \in \br^{3, \varepsilon}_1 \times \br^{3, \varepsilon}_2$, and now define the new inner products $\widehat{\ip}_i$ on $\br^{3,\varepsilon}_i$ by  $\widehat{\ip}_i = 
R_i^{-2} \ip_i$, for $i =1,2$, and $\widehat{\ip}=\widehat{\ip}_1 + \widehat{\ip}_2$.

Relative to the inner product $\widehat{\ip}$, $\tM_i^2(\kappa_i)$  becomes $\tM_i^2(\varepsilon)$, and so $f$ becomes an immersion $\hat f\colon M^3\to \tM_1^2(\varepsilon)\times\tM_2^2(\varepsilon)\subset (\br^{3, \varepsilon}_1 \times \br^{3, \varepsilon}_2, \widehat{\ip})$. Arguing exactly as in the proof of Proposition~\ref{prop:scale}, one can prove the following proposition, which shows that $f$ is given as in~\ref{it:thm14}.

\begin{proposition} \label{prop:scale2}
  The metric induced by $\hat f$ on $M^3$ from $\widehat{\ip}$  has constant sectional curvature $\frac{\varepsilon}{2}$, and the shape operator of $\hat f$ also has rank $2$.
\end{proposition}

\begin{remark}
The preceding discussion applies to any hypersurface $f\colon M^3(\kappa)\to \tM^2(\kappa_1)\times \tM^2(\kappa_2)$ with $\rk S=2$ everywhere. However, we now show that, in particular, if its shape operator $S$ satisfies $S_{12}=0$, then $f$ is given as in Example \ref{ex:3.6}.

We observe that the condition $S_{12}=0$ is equivalent to requiring that $S\cR= \cR S$. 
Since any hypersurface $f\colon M^3(\kappa)\to \tM^2(\kappa_1)\times \tM^2(\kappa_2)$, with $\rk S=2$ everywhere, has a constant product angle function, it follows from \cite[Corollary~5.1]{CT2} that, under the assumption that $S\cR= \cR S$, it is given by \eqref{param}, with $f_1$ and $f_2$ as in \eqref{f1} and \eqref{f2}, respectively, for some $\theta\in \mathbb{R}$. The fact that $M^3$ has constant sectional curvature $\kappa=\frac{\kappa_1\kappa_2}{\kappa_+\kappa_2}$ now implies that
$\theta=\arcsin{\left(\frac{\kappa_1}{\kappa_1+\kappa_2}\right)^{1/2}}$ and that the curvatures $\lambda_1$ and $\lambda_2$ of $\gamma_1$ and $\gamma_2$, respectively, must be constants such that $\lambda_1\lambda_2=(\kappa_1\kappa_2)^{1/2}$, hence it is given as in Example \ref{ex:3.6}.
\end{remark}

\vspace{1ex}
To complete the proof of the theorem, it remains to prove that if a hypersurface $f\colon M^3(\kappa)\to \tM^2(\kappa_1)\times \tM^2(\kappa_2)$, with $\kappa_1, \kappa_2<0$, is as in \ref{it:thm13} in an open subset $\Omega\subset M^3(\kappa)$, then the same must hold on the entire $M^3(\kappa)$.

For a hypersurface $f\colon M^3(\kappa)\to \tM^2(\kappa_1)\times \tM^2(\kappa_2)$, $\kappa_1, \kappa_2<0$, as in~\ref{it:thm13}, the shape operator is given by $S=aI+b\cR$, with $a$ and $b$ as in \eqref{eq:ab}. Therefore, the eigenvalues of $S$ are
$a=\frac{\kappa_1}{[-(\kappa_1+\kappa_2)]^{1/2}}$, $a+b=-\frac{\kappa_2}{[-(\kappa_1+\kappa_2)]^{1/2}}$, and $a+rb$, where $r\in(0,1)$ is an eigenvalue of $\cR$. 
Suppose that the boundary $\partial \Omega$ of $ \Omega$ is nonempty. Then,
at a point $x\in M^3(\kappa)$ in $\partial \Omega$, there must exist a sequence $(x_n)$ in $\Omega$ converging to $x$ such that $a+r(x_n)b$ converges to $0$. Thus, $x$ also belongs to 
to the boundary $\partial \Omega'$ of the subset $\Omega'$ where $\rk S=2$ and $S\cR=\cR S$. On the other hand, at a point $(s, x_1, x_2)\in  \Omega'$, the nonzero principal curvatures are given by \eqref{eq:pc}. For $\theta=\arcsin \left(\frac{\kappa_1}{\kappa_1+\kappa_2}\right)^{1/2}$, we have
$$\frac{\sin\theta}{R_1}= \frac{\kappa_1}{[-(\kappa_1+\kappa_2)]^{1/2}}, \quad \frac{\cos\theta}{R_2}= -\frac{\kappa_2}{[-(\kappa_1+\kappa_2)]^{1/2}},\quad \frac{\cos\theta}{R_1}= \left(- \frac{\kappa_1\kappa_2}{\kappa_1+\kappa_2}\right)^{1/2}= \frac{\sin\theta}{R_2}$$
Thus, setting $K=\left(- \frac{\kappa_1\kappa_2}{\kappa_1+\kappa_2}\right)^{1/2}$ and $d=\lambda_1R_1=\lambda_2R_2$, the nonzero principal curvatures can be written as
\begin{equation*}
\rho_1=-\frac{\kappa_1}{[-(\kappa_1+\kappa_2)]^{1/2}}\frac{-\sinh (Ks)+d\cosh(Ks)}{\cosh(Ks)-d\sinh(Ks)}
=\frac{\kappa_1}{[-(\kappa_1+\kappa_2)]^{1/2}}\frac{\tanh(Ks)-d}{1-d\tanh(Ks)}
\end{equation*}
and
\begin{equation*}
\rho_2=-\frac{\kappa_2}{[-(\kappa_1+\kappa_2)]^{1/2}}\frac{1-d\tanh(Ks)}{\tanh(Ks)-d}.
\end{equation*}
Therefore, for a sequence $(s_n, (x_1)_n, (x_2)_n)\in \Omega'$ converging to $x$, both sequences 
$\frac{\tanh(Ks_n)-d}{1-d\tanh(Ks_n)}$ and $\frac{1-d\tanh(Ks_n)}{\tanh(Ks_n)-d}$ must converge to $1$, which can only occur if 
$s_n\mapsto \infty$. This leads to a contradiction, establishing that no such point $x$ can exist. Consequently, we must have $\Omega=M^3(\kappa)$.

\smallskip
This completes the proof of the theorem.\hfill\qed

\section{Flat hypersurfaces of $\mathbb{R}^k\times \mathbb{H}^{n-k+1}$}\label{ss:flat}

In this last section, inspired by Examples~ \ref{ex:int1} and \ref{ex:int2}, we give a geometric description of the flat hypersurfaces of $\mathbb{R}^k\times \mathbb{H}^{n-k+1}$ for  $1\leq k\leq n-3$. It shows, in particular,  how the complete flat spherical rotational hypersurfaces
of $\mathbb{R}\times \mathbb{H}^n$ in \cite[Theorem~4.4(iii)--(a)]{MaT} arise geometrically. It also provides a global version, for $2\leq k\leq n-3$, of the local classification of  flat hypersurfaces of $\mathbb{R}^k\times \mathbb{H}^{n-k+1}$ obtained in \cite[Theorem~12(iii)]{CT}  for any  $2\leq k\leq n-1$.

First observe that, regarding $\mathbb{R}^k\times \mathbb{H}^{n-k+1}$, $1\leq k\leq n-1$, as a hypersurface of
$\mathbb{R}^k\times \mathbb{R}^{n-k+2,1}=\mathbb{R}^{n+2,1}$, the intersection 
$M^n$ of $\mathbb{R}^k\times \mathbb{H}^{n-k+1}$ with an affine hyperplane of  $\mathbb{R}^{n+2,1}$ orthogonal to a light-like vector $w\in \mathbb{R}^{n+2,1}$ is a flat hypersurface of $\mathbb{R}^k\times \mathbb{H}^{n-k+1}$. Indeed, for $w$ is a normal vector field along the inclusion $i\colon M^n\to \mathbb{R}^{n+2,1}$ such that $\langle \secf_{i}(X,Y), w\rangle=0$ for all $X, Y\in \Gamma(TM^n)$,  and since $N_iM^n(x)$ is a Lorentzian plane for any $x\in M^n$ and $w$ is light-like, it follows that the image subspace of $\secf_{i}(x)$ is spanned by $w$ for every  $x\in M^n$. The Gauss equation of $i$ then implies that $M^n$ is flat. 

The next theorem shows that, conversely, any flat hypersurface of $\mathbb{R}^k\times \mathbb{H}^{n-k+1}$ for $1\leq k\leq n-3$ arises in this way.

\begin{theorem} \label{thm: flatH} 
If $f\colon M^n\to \mathbb{R}^k\times \mathbb{H}^{n-k+1}\subset \mathbb{R}^k\times \mathbb{R}^{n-k+2, 1}=\mathbb{R}^{n+2, 1}
$,
$
1\leq k\leq n-3,
$ 
is a flat hypersurface, then $f(M^n)$ is an open subset of the intersection of $\mathbb{R}^k\times \mathbb{H}^{n-k+1}$ with an affine hyperplane $\mathcal{H}$  of $\mathbb{R}^{n+2, 1}$ orthogonal to a light-like vector $w\in \mathbb{R}^{n+2, 1}$. Moreover:
\begin{enumerate}[label={\rm{(\roman*)}}]
\item\label{it:flat1} if $w\in \mathbb{R}^{n-k+2, 1}$, then $f(M^n)$ is an open subset of the image of an extrinsic product
$$
\id\times g\colon \mathbb{R}^{k}\times \mathbb{R}^{n-k}\to \mathbb{R}^{k}\times \mathbb{H}^{n-k+1},
$$
where 
$
g\colon  \mathbb{R}^{n-k}\to  \mathbb{H}^{n-k+1}
$ 
is a horosphere and  
$
\id \colon \mathbb{R}^{k}\to \mathbb{R}^{k}
$ 
is the identity map;

\smallskip
\item\label{it:flat2} if $w\not\in \mathbb{R}^{n-k+2, 1}$, then $f(M^n)$ is an open subset of the image of an extrinsic product
$$
\id\times g\colon \mathbb{R}^{k-1}\times \mathbb{R}^{n-k+1}\to \mathbb{R}^{k-1}\times(\mathbb{R}\times  \mathbb{H}^{n-k+1}),
$$
where 
$
g\colon  \mathbb{R}^{n-k+1}\to \mathbb{R}\times \mathbb{H}^{n-k+1}
$ 
is an isometric immersion, whose image is a {\rm(}complete flat{\rm)} spherical rotational hypersurface
by  \cite[Theorem~7.2$({\rm iii})-({\rm c})$]{MaT}, and 
$
\id \colon \mathbb{R}^{k-1}\to \mathbb{R}^{k-1}
$ 
is the identity map.
\end{enumerate} 
\end{theorem}
\begin{proof}
 Let 
$
j \colon \mathbb{R}^k\times \mathbb{H}^{n-k+1}\to \mathbb{R}^k\times\mathbb{R}^{n-k+2,1}=\mathbb{R}^{n+2,1}
$ 
be the inclusion,  and let 
$
\tilde f=j \circ f \colon  M^n \to \mathbb{R}^{n+2,1} 
$ 
be its composition with $f$. 
Flatness of $M^n$ implies that $\secf_{\tilde f}(x)$
is a flat bilinear form for any $x \in M^n$, with respect to the Lorentzian inner product in $N_{\tilde f}M^n(x)$. If $\mathcal{I}(\secf_{\tilde f}(x))$ was nondegenerate, that is, $\mathcal{I}(\secf_{\tilde f}(x)) \cap \mathcal{I}(\secf_{\tilde f}(x))^{\perp} = \{0\}$, then
$
\dim  \mathcal{N}(\secf_{\tilde f}(x)) \geq n-2
$ by Lemma~\ref{lem:flatbi}\eqref{it:symim}, 
and since $\mathcal{N}(\secf_{\tilde f}(x))=\ker S(x)\cap \ker \cR(x)$, then
\begin{equation}\label{ineq}
n-2\leq \dim (\ker S(x)\cap \ker \cR(x))\leq \dim  \ker \cR(x) \leq k,
\end{equation} 
a contradiction. Thus, $\mathcal{I}(\secf_{\tilde f}(x))$ is degenerate for every $x\in M^n$. Then, for a choice of one of the two unit normal vectors $\xi$ to $f$ in $\mathbb{R}^k \times \mathbb{H}^{n-k+1}$,  the subspace $\mathcal{I}(\secf_{\tilde f})\cap \mathcal{I}(\secf_{\tilde f})^{\perp}$ is spanned by the light-like vector $\nu=j_*\xi - \nu_2$, where $\nu_2=\tilde{\pi}_2\circ \tilde{f}$, with $\tilde{\pi}_2\colon \mathbb{R}^{n+2,1}\to  \mathbb{R}^{n-k+2,1}$ the orthogonal projection.  Using the first equation in \eqref{eq:nu'}, this implies that 
\begin{equation*}
    0 = \< \secf_{\tilde f}(X,Y),  j_*\xi - \nu_2 \> = \<   \<SX,Y\> j_*\xi - \<\cR X,Y\>\nu_2, j_*\xi - \nu_2 \> = \<SX,Y\> - \<\cR X,Y\> 
\end{equation*}
 for all $X,Y \in \Gamma(TM^n)$, hence $S=\cR$.  Therefore,
 $
\secf_{\tilde f}(X,Y)=\<\cR X,Y\>\nu
$
  for all $X,Y \in \Gamma(TM^n)$. Now, using the second equation in \eqref{eq:nu'}, we have
  $$
\nabla_X^\perp \nu =
\nabla_X^\perp j_*\xi-
\nabla_X^\perp \nu_2\\
=\<X, \zeta\>(-j_*\xi+f_2)\\
= -\<X, \zeta\>\nu.
 $$
Notice that by Lemma~\ref{lem:grad} the vector field  $\zeta$ is a gradient vector field. Taking (locally) a positive $\la\in C^{\infty}(M^n)$ such that 
$\grad(\log\la)= \zeta$ implies that $w=\lambda\nu$ is a constant vector in $\br^{n+2,1}$, for $S^{\tilde f}_w=0$ and 
 $$\nabla_X^\perp w=X(\lambda)\nu+\lambda \nabla_X^\perp \nu=\<\grad \lambda, X\>\nu-\lambda\<X,\zeta\>\nu=0.$$ Therefore,  $\tilde f(M^n)$ is contained in the intersection of $\mathbb{R}^k\times \mathbb{H}^{n-k+1}\subset \br^{n+2,1}$ with an affine (degenerate) hyperplane $\mathcal{H}$ orthogonal to $w$.

 If $w\in \mathbb{R}^{n-k+2, 1}$, then $\tilde{f}(M^n)\subset \mathbb{R}^k\times (\mathbb{H}^{n-k+1}\cap \mathcal{H})$, and \ref{it:flat1} holds. Otherwise,
 $w\in \mathbb{R}\times \mathbb{R}^{n-k+2, 1}$, up to an orthogonal transformation of $\mathbb{R}^{n+2, 1}$, hence
 \[\tilde{f}(M^n)\subset \mathbb{R}^{k-1}\times ((\mathbb{R}\times \mathbb{H}^{n-k+1})\cap \mathcal{H}),\] and \ref{it:flat2} holds for the 
 flat hypersurface $(\mathbb{R}\times \mathbb{H}^{n-k+1})\cap \mathcal{H}\subset \mathbb{R}^{n-k+2, 1}$ of $\mathbb{R}\times \mathbb{H}^{n-k+1}$. 
\end{proof}

 \begin{remark} 
 The case in which $k=1$ and $n=3$ is not covered under the assumption $1\leq k\leq n-3$ of Theorem \ref{thm: flatH}. Following the proof of Theorem \ref{thm: flatH}, we see that the inequalities \eqref{ineq} do not lead to a contradiction in this case, but rather imply that
 $\dim \ker \cR(x)=1=k$ for every $x\in M^3$.
 Then, Proposition~\ref{splits} implies that $f$ is locally decomposable, and, since $M^3$ is flat, $f$ is a vertical cylinder over a flat surface in $\mathbb{H}^3$.
 This corresponds to  \cite[Theorem~7.2(iii)-(a)]{MaT} in the classification of hypersurfaces with constant sectional curvature of $\mathbb{R}\times \mathbb{H}^n$.
 \end{remark}

\section*{Acknowledgements}
We thank Luc Vrancken for sending us an earlier version of \cite{LVWYhyp}  before it was publicly available.
After a preliminary version of this article (which was originally part of the preprint \cite{KNT}) was finished and some of its results were announced in \cite{CT} and \cite{CT2},
we became aware of the new version of \cite{LVWYhyp}, in which an appendix was included with parametrisations, similar to those in Subsection \ref{ss:alternative}, of the
hypersurfaces of dimension $n\geq 4$ and constant sectional curvature of a product of two spaces with constant sectional curvature, but only in the particular case where their
curvatures have the same absolute values. 

The first-named author is thankful to the Department of Mathematical and Physical Sciences at La Trobe University for its warm hospitality during his visit, where this project was initiated.

\end{document}